\documentclass[sigconf]{acmart}
\usepackage{amsmath,amsfonts,amsthm}
\usepackage[bb=boondox]{mathalfa}
\usepackage[algoruled]{algorithm2e}
\usepackage{array}

\usepackage[noabbrev,capitalize]{cleveref}

\crefname{equation}{}{}
\Crefname{equation}{}{}
\usepackage{amsthm}

\theoremstyle{remark} 

\usepackage{tikz}

\newcommand{\set}[1]{\mathcal{#1}} 

\newcommand{\RNum}[1]{\uppercase\expandafter{\romannumeral #1\relax}}
\newcommand{\rNum}[1]{\expandafter\romannumeral #1}

\makeatletter
\newcommand{\pushright}[1]{\ifmeasuring@#1\else\omit\hfill$\displaystyle#1$\fi\ignorespaces}
\newcommand{\pushleft}[1]{\ifmeasuring@#1\else\omit$\displaystyle#1$\hfill\fi\ignorespaces}
\makeatother

\newcolumntype{C}[1]{>{\centering\arraybackslash}p{#1}} 

\AtBeginDocument{%
  \providecommand\BibTeX{{%
    \normalfont B\kern-0.5em{\scshape i\kern-0.25em b}\kern-0.8em\TeX}}}

\setcopyright{acmcopyright}

\copyrightyear{2023}
\acmYear{2023}
\setcopyright{acmlicensed}\acmConference[e-Energy '23]{The 14th ACM International Conference on Future Energy Systems}{June 20--23, 2023}{Orlando, FL, USA}
\acmBooktitle{The 14th ACM International Conference on Future Energy Systems (e-Energy '23), June 20--23, 2023, Orlando, FL, USA}
\acmPrice{15.00}
\acmDOI{10.1145/3575813.3597356}
\acmISBN{979-8-4007-0032-3/23/06}

\begin{document}

\title[Data-Driven Inverse Optimization for Marginal Offer Price Recovery in Electricity Markets]{Data-Driven Inverse Optimization for Marginal Offer Price Recovery in Electricity Markets}

\author{Zhirui Liang}
\affiliation{%
  \department{Department of Electrical and Computer Engineering}
  \institution{Johns Hopkins University}
  \city{Baltimore}
  \country{United States}}
\email{zliang31@jhu.edu}

\author{Yury Dvorkin}
\affiliation{%
  \department{Ralph O’Connor Sustainable Energy Institute}
  \department{Department of Electrical and Computer Engineering}
  \department{Department of Civil and System Engineering}
  \institution{Johns Hopkins University}
  \city{Baltimore}
  \country{United States}}
  \email{ydvorki1@jhu.edu}

\renewcommand{\shortauthors}{Z. Liang and Y. Dvorkin}

\begin{abstract}
  This paper presents a data-driven inverse optimization (IO) approach to recover the marginal offer prices of generators in a wholesale energy market. By leveraging underlying market-clearing processes, we establish a closed-form relationship between the unknown parameters and the publicly available market-clearing results. Based on this relationship, we formulate the data-driven IO problem as a computationally feasible single-level optimization problem. The solution of the data-driven model is based on the gradient descent method, which provides an error bound on the optimal solution and a sub-linear convergence rate. We also rigorously prove the existence and uniqueness of the global optimum to the proposed data-driven IO problem and analyze its robustness in two possible noisy settings. The effectiveness of the proposed method is demonstrated through simulations in both an illustrative IEEE 14-bus system and a realistic NYISO 1814-bus system.
\end{abstract}

    

\received{10 February 2023}

\maketitle

\section{Introduction}
\label{sec:Introduction}
The US power sector has undergone a transformation from a monopoly to a liberalized system since the 1980s, in order to promote competition and increase efficiency \cite{li2011modeling}. This change has opened up opportunities for power producers to make more profits, but also increases the risk of not being selected for dispatch if their offer price surpasses that of their competitors. Knowing the market, including the prices offered by competitors, allows power producers to maximize their profits and minimize their risk exposure \cite{eckbo2009bidding}. In parallel, open-access policies (e.g., FERC Order 888) have authorized the release of some market data, including selected transmission network parameters and market-clearing results \cite{helman2006market}, which allows for data mining and reverse-engineering of the unreleased market data. Building on previous work that has examined the recovery of rival marginal offer prices from the market-clearing results through inverse optimization \cite{ruiz2013revealing,chen2017strategic,chen2019learning,birge2017inverse}, this paper also explores the price recovery problem in electricity markets but with a focus on designing performance guarantees for convergence, robustness, and uniqueness of the global optimum in inverse optimization problems.

\subsection{Literature Review}
\label{subsec:Literature_review}
An inverse optimization (IO) problem seeks to recover unknown parameters in the objective function \cite{keshavarz2011imputing} or constraints \cite{chan2020inverse} of a forward optimization (FO) problem. 
The FO could be a linear program \cite{ahuja2001inverse}, a conic program \cite{iyengar2005inverse}, a mixed-integer program \cite{ahmadian2018algorithms}, a linearly constrained separable convex program \cite{zhang2010inverse}, or more generally, a variational inequality function \cite{thai2018imputing,bertsimas2015data}. 

The classic IO model assumes that there exists a single set of parameters that make the given solutions optimal. However, in reality, noise in observations can prevent the existence of parameters that make all solutions strictly optimal. This noise can come from various sources, such as (i) measurement errors during the data collection process, (ii) mismatches between the forward model and the actual underlying decision-making process, and (iii) bounded rationality (i.e., when the decision maker settles for sub-optimal results due to cognitive or computational limitations) \cite{aswani2018inverse,mohajerin2018data}. Applying the classic IO in noisy settings can lead to infeasibility or uninformative solutions, such as a trivial zero cost vector \cite{chan2019inverse}. Therefore, the goal of IO in noisy settings is to minimize the fitting errors between the model and data to make the given solutions “approximately” optimal \cite{thai2018imputing,mohajerin2018data,aswani2018inverse,chan2019inverse}.

In deterministic settings, IO can precisely recover the parameters of the FO problem based on a single observation. However, in noisy settings, multiple observations are required to minimize the empirical loss of the IO problem. This type of IO, which can handle multiple observations, is also known as data-driven IO \cite{mahmoudzadeh2022learning,mohajerin2018data}. Solving data-driven IO problems can become difficult with an increase in variables as more training data is incorporated. To tackle this challenge, researchers have proposed to use machine learning (ML) methods instead of conventional data-driven IO models \cite{barmann2017emulating,dong2018generalized,tan2019deep,bian2022demand}. 

The objective of IO is similar to those of some ML techniques, especially neural networks, as they both aim to estimate the unknown parameters of a model based on available observations. However, the key difference is that the underlying model in IO corresponds to the actual FO problem and its parameters can be interpreted, while in ML, the unknown parameters (such as weights and biases in neural networks) often lack a straightforward explanation. To incorporate physical laws into ML models, researchers have developed physics-informed ML methods \cite{karniadakis2021physics}. Essentially, if the FO problem with unknown parameters is transformed into a differentiable layer in the neural network, these unknown parameters can be learned through training (similar to other parameters in the network) \cite{amos2017optnet,agrawal2019differentiable,butler2022efficient,ferber2020mipaal}. 
For instance, OptNet is a network architecture that integrates a quadratic optimization program as a layer within an end-to-end neural network \cite{amos2017optnet}. In \cite{bian2022demand}, OptNet is used to determine the unknown parameters in a demand response model. The estimation of unknown parameters can also be achieved through other ML methods. For example, \cite{barmann2017emulating,dong2018generalized} reveal the parameters through an online learning process, while \cite{tan2019deep} presents a deep inverse optimization model that formulates the parameter recovery problem as a deep learning model. The fundamental assumption behind these ML models is that the FO model is differentiable with respect to the unknown parameters, which allows for the computation of gradients and the training based on (stochastic) gradient descent methods.

To compare IO and non-physics-informed ML, \cite{iraj2021comparing} evaluated their performance in imputing a convex objective function and found that IO is more suitable for problems with smaller training sets or highly complex objectives/constraints, while ML can handle larger training sets and is less sensitive to noisy training data. However, potential drawbacks of ML, such as its tendency to diverge in extreme situations, restrict its use in critical decision-making. Hence, to gain the trust of decision makers in using ML, formal performance guarantees need to be established for ML models \cite{avin2021filling}.

IO has various applications across fields, including in power systems and power markets. 
To maximize the profit of market participants in the day-ahead bidding process, IO has been used to recover the piece-wise linear cost function of generators based on historical market-clearing results in \cite{ruiz2013revealing,chen2017strategic,chen2019learning}. Similarly, \cite{risanger2020inverse} also designs a IO problem based on market-clearing results, but the forward problem is an equilibrium model in an oligopolistic electricity market, aimed at determining if the market structure matches the observed outcomes. While \cite{ruiz2013revealing,chen2017strategic,chen2019learning,risanger2020inverse} assume that all parameters except for offer prices are known, \cite{birge2017inverse} uses IO to recover market structures that are not disclosed to market participants, such as parameters of transmission lines. 
Alternatively, IO has been used to estimate the parameters of demand response in power systems in studies such as \cite{saez2016data,saez2017short,kovacs2021inverse}. In \cite{saez2016data}, IO is recast as a bi-level program and solved heuristically. \cite{saez2017short} deals with a non-convex IO model and develops a two-step heuristic approach to solve it. In \cite{kovacs2021inverse}, IO is transformed into a quadratically constrained quadratic program and solved through successive linear programming.

The aim of this paper is to develop a data-driven IO method for recovering the marginal offer prices of generators from historical market-clearing results. The US power systems and markets have the following unique properties that distinguish this IO problem from other application scenarios.

First, the structures and most parameters of power markets are \textit{transparent} \cite{NYISO_RNA_report}, and the market-clearing results, including schedules and prices, are \textit{promptly and accurately published} \cite{NYISO_price,NYISO_market_clearing_results}. This allows for the observation of both primal variables (schedules) and some dual variables (prices) in the forward problem, giving our price recovery problem an advantage over other IO problems where only primal solutions can be observed.

Second, the market operation relies on the use of \textit{locational marginal prices (LMPs)} which represent the incremental costs of providing the next megawatt of energy at different locations within the power system. However, the offer price of a generator is only reflected in the LMPs when it is the \textit{marginal} generator at that location. Hence, a single observation is insufficient to recover the offer prices of all generators, requiring the use of data-driven IO to explore multiple scenarios where different generators are marginal.

Third, the offer price of a generator tend to be \textit{stable} over time. In day-ahead markets, generators are paid according to LMPs instead of their offer prices. To secure market participation, they typically set the offer prices close to or slightly above their generation costs. However, the offer price of a generator may vary due to external factors like fuel prices and weather, causing the observations from the market being the optimal solutions to the same forward model but with different cost vectors.

Finally, the noise in market-clearing results is \textit{sparse and traceable}. According to \cite{LMP_slides}, only a small percentage (less than 0.1\%) of real-time intervals have resulted in price correction since 2009, indicating that random errors in published price data are uncommon. Another source of noise is the mismatch between the forward model studied in this paper (i.e., the DC optimal power flow model) and the actual market-clearing model which takes into account more constraints and may involve ad-hoc manipulations by market operators. These noises may affect the accuracy of IO results.

In summary, the transparent market structure and abundant market data make the data-driven IO feasible. However, the noise in historical data and the fluctuating nature of offer prices pose challenges to the IO problem. We plan to tackle these challenges by solving the data-driven IO problem using the gradient descent method to effectively utilize the training data and ensure robustness in noisy settings. Our focus is also on establishing performance guarantees for the proposed approach.

\subsection{Contributions}
\label{subsec:contributions}
After reviewing the related work on IO in Section~\ref{sec:IO_review} and formulating the FO problem in Section~\ref{sec:Forward_model}, our contributions in the rest of the paper are: 
\begin{itemize}
    \item We propose a data-driven IO model for marginal offer price recovery in wholesale power markets. Our method differs from previous studies such as \cite{ruiz2013revealing,chen2019learning} since it incorporates the unknown cost vector directly into the objective function, rather than relying solely on the primal or dual variables of the FO problem. This design results in a computationally manageable single-level optimization problem, compared to the bi-level programs in \cite{aswani2018inverse,mohajerin2018data}.
    \item Unlike previous work that solves the data-driven IO using off-the-shelf solvers \cite{ruiz2013revealing,chen2017strategic,chen2019learning,risanger2020inverse,birge2017inverse}, we solve it using gradient descent and derive the corresponding error bound. Our method allows for proving the convergence to the global optimum with a sub-linear rate, whereas the method in \cite{bian2022demand} only guarantees convergence in the absence of inequality constraints in the FO problem. Furthermore, we rigorously prove the existence and uniqueness of the global optimum to the data-driven IO problem, which are generally missing in other studies.
    \item Different from \cite{ruiz2013revealing,chen2017strategic,chen2019learning,birge2017inverse} which accomplish the offer price recovery using IO but neglect the impact of noise in observations, our approach considers potential sources of errors in market-clearing results and evaluates robustness of the data-driven IO in two noisy settings. The validity and robustness of the proposed data-driven IO method are demonstrated through simulations in the IEEE 14-bus system and the 1814-bus NYISO transmission network.
\end{itemize}
While the proposed data-driven IO model focuses on the price recovery problem in day-ahead markets, it can also be adapted to applications in real-time and ancillary service markets by appropriately changing constraints (while preserving linearity) and training data.

\subsection{Notations}
\label{subsec:notations}
The following notations are used in the paper: $I_{n}$ is an identity matrix of size $n \times n$ with ones on the main diagonal and zeros elsewhere, $J_{m,n}$ is an all-ones matrix of size $m \times n$, $O_{m,n}$ is an all-zeros matrix of size $m \times n$, $D(\cdot)$ creates a diagonal matrix from a vector, and $\left\| \cdot \right\|_p$ represents the $\ell_p$-norm.

\section{Overview of Inverse Optimization}
\label{sec:IO_review}
In this section, we describe a general linear optimization model as the FO problem, formulate the corresponding IO models in deterministic and noisy settings, extend the single-observation IO to the data-driven IO, and compare the benefits and drawbacks of various IO formulations.

\subsection{Linear FO Problem}
\label{subsec:forward_model}
Let $x \in \mathbb{R}^n$, $w \in \mathbb{R}^n$, $A \in \mathbb{R}^{m \times n}$, and $b \in \mathbb{R}^m$. We define our linear forward optimization (FO) problem as:
\begin{subequations}
\begin{align}
   \textbf{FO}(w): &\min_{x}\ w^T x \label{FO:obj}\\
&\text{s.t. } \ 
    (\xi): Ax \ge b,  \label{FO:constraint}
\end{align}%
\label{mod:FO}%
\end{subequations}%
\allowdisplaybreaks[0]%
where $\xi \in \mathbb{R}^m$ is the vector of dual variables associated with $Ax \ge b$. Let $\set{X}(w)$ be the set of feasible solutions to \textbf{FO}($w$) and $\set{X}^{OPT}(w)$ be the set of optimal solutions to \textbf{FO}($w$). We assume that $\set{X}(w)$ and $\set{X}^{OPT}(w)$ are non-empty.

\subsection{Deterministic IO Formulation}
\label{subsec:deterministic_IO}
Let $x^0 \in \mathbb{R}^n$ denote an observed solution to \textbf{FO}($w$) in \eqref{mod:FO}. In a deterministic setting, there exists a cost vector $w$ such that $x^0$ is an optimal solution to \textbf{FO}($w$). Therefore, given $A$, $b$ and $x^0$, the goal is to find $w$ such that $x^0 \in \set{X}^{opt}(w)$. Note that there may exist multiple $w$ that can achieve this goal \cite{keshavarz2011imputing}. For example, if $x^0 \in \set{X}^{opt}(w)$, then $x^0 \in \set{X}^{opt}(G(w))$ also holds, where $G(\cdot)$ is any convex increasing function. The uniqueness of $w$ can be guaranteed by normalization and regulation, e.g., by requiring $\left\| w \right\| = 1$. Further, if \textbf{FO}($w$) is a linear program, its inverse problem \textbf{IO}($x^0$) is also a linear program  \cite{ahuja2001inverse}.

One way to evaluate the optimality of the estimated $w$ is via the optimality conditions of \textbf{FO}($w$), which can be satisfied exactly by $x^0$ when the optimal $w$ is found. Therefore the loss function of \textbf{FO}($w$) is trivial and can be set to 0. The optimality primal-dual (PD) conditions of \textbf{FO}($w$) can be formulated as \cite{ahuja2001inverse}:
\begin{subequations}
\begin{align}
   \textbf{IO-PD}(x^0): \min_{w,\xi} &\ 0 \label{IO_deter:obj}\\
\text{s.t. } \ 
    & \left\| w \right\| = 1  \label{IO_deter:norm} \\
    & Ax^0 \ge b \label{IO_deter:prim_feasibility} \\
    & A^T \xi = w  \label{IO_deter:dual_feasibility_1} \\
    & \xi \ge O  \label{IO_deter:dual_feasibility_2} \\
    & w^T x^0 = b^T \xi,  \label{IO_deter:strong_duality}
\end{align}%
\label{mod:IO_deter_dual}%
\end{subequations}%
\allowdisplaybreaks[0]%
where \cref{IO_deter:norm} ensures a unique optimal solution, \cref{IO_deter:prim_feasibility} embodies primal feasibility,  \cref{IO_deter:dual_feasibility_1} and \cref{IO_deter:dual_feasibility_2} denote dual feasibility, and \cref{IO_deter:strong_duality} represents strong duality. Note that \cref{IO_deter:prim_feasibility} can be omitted since it does not depend on any primal or dual variables. 

Alternatively, the optimality conditions of \textbf{FO}($w$) can also be formulated as the Karush-Kuhn-Tucker (KKT) conditions \cite{keshavarz2011imputing}:
\begin{subequations}
\begin{align}
   \textbf{IO-KKT}(x^0): \min_{w,x,\xi} &\ 0 \label{FO_kkt:obj}\\
\text{s.t. } \
   & \text{\cref{IO_deter:norm,IO_deter:prim_feasibility,IO_deter:dual_feasibility_1,IO_deter:dual_feasibility_2}} \nonumber \\
   & D(\xi) (b-Ax^0)=O, \label{FO_kkt:comp}
\end{align}%
\label{mod:IO_deter_kkt}%
\end{subequations}%
\allowdisplaybreaks[0]%
where \cref{FO_kkt:comp} is complementary slackness. Note that \cref{IO_deter:dual_feasibility_1} is referred to as the stationarity constraint in the KKT constraints. The \textbf{IO-PD} and \textbf{IO-KKT} formulations in \cref{mod:IO_deter_dual} and \cref{mod:IO_deter_kkt} are equivalent, i.e.,
\begin{equation}
    x^0 \in \set{X}^{opt}(w) \Leftrightarrow  \text{\cref{IO_deter:norm,IO_deter:prim_feasibility,IO_deter:dual_feasibility_1,IO_deter:dual_feasibility_2,IO_deter:strong_duality} hold} \Leftrightarrow  \text{\cref{IO_deter:norm,IO_deter:dual_feasibility_1,IO_deter:prim_feasibility,FO_kkt:comp,IO_deter:prim_feasibility,IO_deter:dual_feasibility_2} hold}, \nonumber
\end{equation}
and their feasibility is guaranteed \cite{chan2019inverse}.

\subsection{Generalized IO (GIO) in Noisy Settings}
\label{subsec:GIO}
In noisy settings, it is assume that the observed solution $x^0$ is feasible but not necessarily optimal in \textbf{FO}($w$), i.e., $x^0 \in \set{X}(w)$. Thus, the optimality conditions in \cref{mod:IO_deter_dual} or \cref{mod:IO_deter_kkt} may not be exactly met, and the objective of IO becomes finding the vector $w$ such that $x^0$ is an “$\epsilon$-optimal” solution, i.e., the choice of $w$ enables $x^0$ to approximately satisfy the optimality conditions of \textbf{FO}($w$). This IO problem is referred to as generalized IO (GIO) in \cite{chan2019inverse} and there are two GIO models reported in the literature.

Based on the \textbf{IO-PD} formulation, GIO can be formulated by adding a slack variable $\epsilon_{pd}$ to relax the constraints of primal feasibility \cref{IO_deter:prim_feasibility} and strong duality \cref{IO_deter:strong_duality} \cite{chan2019inverse,bertsimas2015data,risanger2020inverse}, as given by:
\begin{subequations}
\begin{align}
   \textbf{GIO-PD}(x^0): \min_{w,\xi,\epsilon_{pd}} &\ \left\|\epsilon_{pd} \right\|_p \label{GIO_pd:obj}\\
\text{s.t. } \ 
    &\text{\cref{IO_deter:dual_feasibility_1,IO_deter:dual_feasibility_2,IO_deter:norm}} \nonumber \\
    & A(x^0 +\epsilon_{pd}) \ge b \label{GIO_pd:primal_feasibility} \\
    & w^T (x^0 +\epsilon_{pd}) = b^T \xi. \label{GIO_pd:strong_duality}
\end{align}%
\label{mod:GIO_pd}%
\end{subequations}%
\allowdisplaybreaks[0]%
It is proved in \cite{chan2019inverse} that \cref{GIO_pd:primal_feasibility} can be omitted since it is automatically satisfied by the optimal solution to \textbf{GIO-PD}($x^0$) without \cref{GIO_pd:primal_feasibility}.

Alternatively, GIO can also be formulated based on the \textbf{IO-KKT} formulation by relaxing the stationarity conditions in \cref{IO_deter:dual_feasibility_1} and complementary slackness in \cref{FO_kkt:comp} with slack variables $\epsilon_{stat}$ and $\epsilon_{comp}$ \cite{keshavarz2011imputing,saez2016data}, respectively:
\begin{subequations}
\begin{align}
   \textbf{GIO-KKT}(x^0): \min_{w,\xi,\epsilon_{stat},\epsilon_{comp}} &\ \left\|\epsilon_{stat}\right\|_p + \left\|\epsilon_{comp}\right\|_p \label{GIO_kkt:obj}\\
\text{s.t. } \ 
    & \text{\cref{IO_deter:norm,IO_deter:prim_feasibility,IO_deter:dual_feasibility_2}} \nonumber \\
    & w - A^T \xi = \epsilon_{stat} \label{GIO_kkt:stat}\\
    & D(\xi) (b-Ax^0)= \epsilon_{comp}. \label{GIO_kkt:comp}
\end{align}%
\label{mod:GIO_kkt}%
\end{subequations}%
\allowdisplaybreaks[0]%
If the observed solution is not feasible, i.e., $x^0 \notin \set{X}(w)$, we can add slack variables $\epsilon_{prim}$ and $\epsilon_{dual}$ to the primal feasibility \cref{IO_deter:prim_feasibility} and the dual feasibility \cref{IO_deter:dual_feasibility_2}, respectively, instead of forcing them to hold.

Both forms of GIO are typically straightforward to solve. For example, \cite{chan2019inverse} derives the closed-form solutions to \textbf{GIO-PD} under different forms of loss functions, while \cite{keshavarz2011imputing} proves that \textbf{GIO-KKT} in \cref{mod:GIO_kkt} is a finite-dimensional convex optimization problem which can be solved efficiently with off-the-shelf solvers.

\subsection{Nominal IO (NIO) in Noisy Settings}
\label{subsec:NIO}
The study in \cite{aswani2018inverse} shows that \textbf{GIO-PD} and \textbf{GIO-KKT} are statistically inconsistent, meaning that their results may converge to incorrect values even with a large amount of training data. In response, \cite{ghobadi2018robust} introduces the nominal IO (NIO), in which the objective is to minimize the difference between the observed solution $x^0$ and the optimal solution of \textbf{FO}($w$).
\begin{align}
   \textbf{NIO}(x^0): \min_{w,x} \ \left\|x^0 - {\rm{arg}} \min_{x} \ \textbf{FO}(w)\right\|_p, \label{mod:NIO_pred_bi}
\end{align}
Note that \cref{mod:NIO_pred_bi} is a bi-level program because generally there is no closed-form solution to $\min_x$ \textbf{FO}($w$). To reduce \cref{mod:NIO_pred_bi} to a single-level program, $\min_x$ \textbf{FO}($w$) must be substituted with either the PD constraints in \cref{mod:IO_deter_dual} or the KKT constraints in \cref{mod:IO_deter_kkt}. For example, the NIO formulation based on the PD constraints is given by:
\begin{subequations}
\begin{align}
   \textbf{NIO-PD}(x^0): \min_{w,x,\xi} &\ \left\|x^0-x \right\|_p \label{NIO_pred:obj}\\
\text{s.t. } \ 
    &\text{\cref{IO_deter:norm,IO_deter:dual_feasibility_1,,IO_deter:dual_feasibility_2}} \nonumber \\
    & Ax \ge b \label{NIO:prim_feasibility} \\
    & w^T x = b^T \xi. \label{NIO:strong_duality}
\end{align}%
\label{mod:NIO_pred}%
\end{subequations}%
\allowdisplaybreaks[0]%
Similarly, the NIO based on KKT constraints can be formulated. 
The distinction between the GIO and NIO formulations lies in the choice of loss functions that measure the discrepancy between predictions and observations. In the GIO formulation, the loss is measured by the slackness required to make the observed solution satisfy the approximate optimality conditions under the estimated parameters. In the NIO formulation, the loss is the distance between the observed solution and the optimal solution to FO under the estimated parameters. Moreover, the NIO formulation is statistical consistent \cite{aswani2018inverse} and more robust to noise \cite{thai2018imputing}, but it is a bi-level program and thus NP-hard \cite{aswani2018inverse}, while the GIO formulations are relatively easier to solve.

\subsection{Data-driven IO Formulation}
\label{subsec:IO_sum}
The IO formulations presented in \cref{mod:IO_deter_dual,mod:IO_deter_kkt,mod:GIO_pd,mod:GIO_kkt,mod:NIO_pred_bi,mod:NIO_pred} are based on a single observation $x^0$, but they can be extended to the data-driven form to accommodate multiple observations. For instance, given $N$ observed solutions to FO $\{x^0_i\}_{i \in \set{I}}$ and the corresponding parameters $\{A_i,b_i\}_{i \in \set{I}}$, the NIO in \cref{mod:NIO_pred} can be adapted as follows:
\begin{align}
   \min_{w,\{{x_i}\}_{i \in \set{I}}} \ \sum\nolimits_{i \in \set{I}} \left\|x_i^0 - {\rm{arg}} \min_{x_i} \ \textbf{FO}(w,A_i,b_i)\right\|_p, \label{mod:NIO_pred_multi}
\end{align}
which is a more computationally demanding bi-level problem compared to \cref{mod:NIO_pred}. This is because the number of variables in \cref{mod:NIO_pred_multi} is approximately $N$ times larger than that in \cref{mod:NIO_pred}.

\section{Forward Problem Formulation}
\label{sec:Forward_model}
In this section, we present the standard US power market-clearing procedure, formulate the day-ahead market-clearing model, and derive the expression for locational marginal prices (LMPs) based on the Lagrangian and KKT conditions of the market-clearing model.

\subsection{Energy Market and Its Settlement Process}
\label{subsec:bidding}
In the US, the operation of power systems and power markets are managed by the designated independent system operators (ISOs). ISOs ensure the safe and reliable operation of power systems and regulate markets for various products such as capacity, energy, and ancillary services \cite{user_guide}. Fig.~\ref{fig:power_market} illustrates the typical two-settlement process in energy markets, including the day-ahead market (DAM) and real-time market (RTM). The clearing of DAM is based on security-constrained unit commitment (SCUC) and security-constrained economic dispatch (SCED), which determines the least-cost generator schedules and hourly LMPs for the next day based on day-ahead generation offers, load bids, as well as the forecast of load and renewable energy source (RES) generation. The clearing of RTM is based on real-time commitment (RTC) and real-time dispatch (RTD). The commitment of units is predominantly determined in the DAM, making it a crucial aspect for power suppliers. However, the actual scheduling may diverge from the DAM schedule due to changes in operating conditions, influences of additional RT generation offers, and variations in actual load. This paper focuses on recovering the DA offer prices of generators from the DA market-clearing results.

\begin{figure}[!t]
    \centering
    \includegraphics[width=0.8\linewidth]{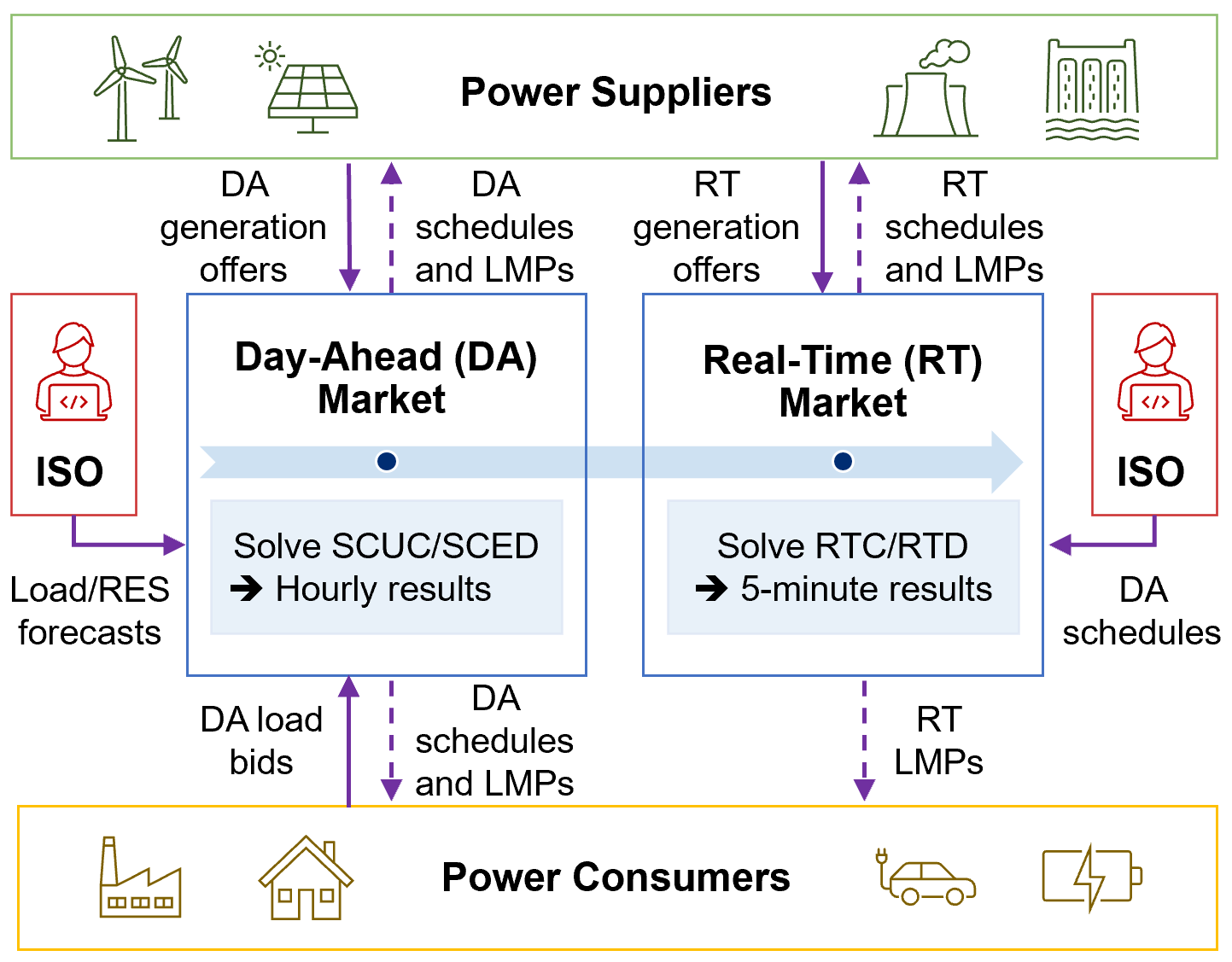}
    \caption{Two-settlement process in US energy markets.}
    \label{fig:power_market}
\end{figure}

\subsection{Day-Ahead Market-Clearing Model}
\label{subsec:SCUC}
The actual SCUC model employed by ISOs is complex and involves multiple solution steps, as described in \cite{DA_scheduling_manual}, and it is simplified to a representative SCUC model in some papers \cite{fu2007fast}. 
Following similar simplifications, we formulate the following UC model for a system with $n$ generators, $m$ nodes and $l$ lines:
\begin{subequations}
\begin{align}
   \textbf{UC}: & \min_{u,x}\ c^T x \label{UC:obj}\\
\text{s.t. } \ 
   & J_{n,1}^T x= J_{m,1}^T e  \label{UC:power_balance}\\
   & x \leq  u x^{\max} \label{UC:pg_max}\\
   & u x^{\min} \leq x \label{UC:pg_min}\\
   & \Phi (S x -e) \leq f^{\max} \label{UC:flow_max} \\
   & - f^{\max} \leq \Phi (S x -e),  \label{UC:flow_min}
\end{align}%
\label{mod:UC}%
\end{subequations}%
\allowdisplaybreaks[0]%
where \cref{UC:obj} minimizes system operation cost using a linear cost function, \cref{UC:power_balance} ensures the power balance in the system, \cref{UC:pg_min} and \cref{UC:pg_max} limit the output of generators to their technical bounds $x^{\min}$ and $x^{\max}$, \cref{UC:flow_min} and \cref{UC:flow_max} limit the power flow on each line to the maximum transmission capacity $f^{\max}$. Other notations used in \cref{mod:UC} are explained in Table~\ref{tab:notations}. 

Model \cref{mod:UC} is a mixed-integer linear program (MILP) due to the presence of binary variables $u$. Although non-convex, it can be solved by modern solvers such as CPLEX or Gurobi. To obtain the dual variables for marginal price calculation, we need to convert \cref{mod:UC} into an equivalent convex linear program (LP) by replacing binary variables to real variable $u\in \mathbb{R}^n$ and setting $u=u^*$, where $u^*$ is the optimal commitment results obtained by solving \cref{mod:UC} with a MILP solver. This approach follows the results in \cite{o2005efficient}, where the authors show that the optimal solution of the MILP is equal to the optimal solution of the LP when $u=u^*$ is enforced in LP. 
The resulting LP is the following DC optimal power flow (DCOPF) model:
\begin{subequations}
\begin{align}
   \textbf{DCOPF}: & \min_{x}\ c^T x \label{DCOPF:obj}\\
\text{s.t. } \ 
   &(\lambda): J_{n,1}^T x= J_{m,1}^T e  \label{DCOPF:power_balance}\\
   &(\alpha): x \leq  u^*x^{\max} \label{DCOPF:pg_max}\\
   &(\beta): u^*x^{\min} \leq x \label{DCOPF:pg_min}\\
   &(\mu): \Phi (S x -e) \leq f^{\max} \label{DCOPF:flow_max} \\
   &(\nu): - f^{\max} \leq \Phi (S x -e).  \label{DCOPF:flow_min}
\end{align}%
\label{mod:DCOPF}%
\end{subequations}%
\allowdisplaybreaks[0]%
Greek letters in parentheses in \cref{DCOPF:power_balance,DCOPF:pg_min,DCOPF:pg_max,DCOPF:flow_min,DCOPF:flow_max} denote Lagrangian multipliers (dual variables) of the respective constraints. Specifically, $\lambda$ denotes the energy price at the reference node in the power system.
The KKT optimality conditions of \cref{mod:DCOPF} are:
\begin{subequations}
\begin{align}
   & \text{\cref{DCOPF:power_balance,DCOPF:pg_min,DCOPF:pg_max,DCOPF:flow_min,DCOPF:flow_max}} \nonumber\\
   &c - \lambda J_{n,1} +\alpha -\beta+(\Phi S)^T \mu- (\Phi S)^T \nu = O_{n,1} \label{matrix_kkt:pg}\\
   & D(\alpha) (x-u^*x^{\max})=O_{n,1} \label{matrix_kkt:alpha}\\
   & D(\beta) (u^*x^{\min}-x)=O_{n,1} \label{matrix_kkt:beta}\\
   & D(\mu) (\Phi S x - \Phi e-f^{\max})=O_{l,1}  \label{matrix_kkt:mu}\\
   & D(\nu) (-\Phi S x + \Phi e-f^{\max})=O_{l,1}  \label{matrix_kkt:nu} \\
   & \alpha \ge O_{n,1},\ \beta \ge O_{n,1},\ \mu \ge O_{l,1},\ \nu \ge O_{l,1}.
\end{align}%
\label{mod:matrix_KKT}%
\end{subequations}%
\allowdisplaybreaks[0]%
The Lagrangian function of \cref{mod:DCOPF} is:
\begin{align}
\mathcal{L}(x,\lambda,\alpha,\beta,\mu,\nu) &=c^T x+ \lambda(J_{m,1}^T e- J_{n,1}^T x) + \alpha^T(x-u^*x^{\max})\nonumber \\ 
&  + \beta^T( u^*x^{\min}-x) + \mu^T(\Phi S x - \Phi e-f^{\max}) \nonumber \\ 
&  + \nu^T(-\Phi S x + \Phi e-f^{\max}).
\label{matrix_lagrangian}
\end{align}
We assume that market participants have limited information about the actual market-clearing model used by the ISOs, i.e., they are only aware of the DCOPF model in its simplest form. Therefore, the DCOPF model in \cref{mod:DCOPF} will be used as the FO model in Section~\ref{sec:Proposed_model}. If some market participants have more knowledge about the actual (ground truth) market-clearing model, that is knowledge exceeding the DCOPF model in \cref{mod:DCOPF}, the IO model proposed in Section~\ref{sec:Proposed_model} could be upgraded to accommodate the FO problem with more constraints (e.g., linearized ACOPF).

\subsection{LMP Derivation from the DCOPF Model}
\label{subsec:LMP}
The use of LMPs is the cornerstone of market operation \cite{LMP_slides}.
The marginal energy price, represented by the dual variable $\lambda$ in \cref{DCOPF:power_balance}, is the LMP at all nodes when the transmission capacity is unlimited. It reflects the cost of serving the next increment of load in the most efficient way. Based on \cref{matrix_kkt:pg}, we can express $\lambda$ as:
\begin{equation}
   \lambda = J_{1,n} \big[c + \alpha -\beta+(\Phi S)^T \mu- (\Phi S)^T \nu \big]  \label{eq:reference_price}
\end{equation}
However, LMPs usually vary between different locations since the transmission constraints obstruct the flow of the least-cost electricity to every part of the system, resulting in congestion costs.
The vector of LMPs ($\omega$) can be derived based on \cref{matrix_lagrangian} as:
\begin{equation}
   \textbf{LMPs}: \omega \equiv \frac{\partial \mathcal{L}}{\partial e} = J_{m,1} \lambda - \Phi^T \mu + \Phi^T \nu,
\label{eq:lmp_without_c}
\end{equation}
where $\lambda$ is the LMP at the reference node (so $\lambda$ is a component of vector $\omega$), and $\Phi^T (\nu-\mu)$ is congestion cost. Another expression of LMPs can be derived by substituting \cref{eq:reference_price} into \cref{eq:lmp_without_c}, as:
\begin{align}
   \omega & = J_{m,n} \big[c + \alpha -\beta+(\Phi S)^T (\mu - \nu) \big] + \Phi^T (\nu - \mu) \nonumber \\
   & = J_{m,n} (c + \alpha - \beta)+ (J_{m,n} S^T - I_m)\Phi^T (\nu - \mu)
   \label{eq:lmp_with_c}
\end{align}
In power markets, while the offer prices of generators (elements in cost vector $c$) are confidential, the market-clearing results including the optimal schedules ($x$) and prices ($\lambda$ and $\omega$) are disclosed. Given the relationships between the unknown cost vector $c$ and the publicly available data $\omega$ and $\lambda$ provided by \cref{eq:reference_price,eq:lmp_with_c}, it is possible to recover $c$ from the historical values of $\omega$ and $\lambda$ using inverse optimization, provided that the schedules (i.e., $e$, $u$, $x$) and the system parameters (i.e., $x^{\max}$, $x^{\min}$, $f^{\max}$, $\Phi$, $S$) are known. Note that a single observation is not enough to recover all the offer prices, as the offer price of a generator is reflected in LMPs only when the generator is marginal. Hence, data-driven IO has to be used and its performance in noisy settings should be evaluated.

\begin{table}[!t]
  \centering
  \caption{Notations in models \cref{mod:UC} and \cref{mod:DCOPF}} 
    \begin{tabular}{cc}
    \toprule
    Notation 	& Meaning	 \\
    \midrule
    $c \in \mathbb{R}^n$ & offer prices of generators \\
    $x \in \mathbb{R}^n$ & output power of generators \\
    $u \in \mathbb{R}^n$ & commitment status of generators \\
    $e \in \mathbb{R}^m$ & net load (demand minus renewable injections)\\
    $x^{\max} \in \mathbb{R}^n$ & upper bound of $x$ \\
    $x^{\min} \in \mathbb{R}^n$ & lower bound of $x$ \\
    $\Phi \in \mathbb{R}^{l \times m}$ & power transfer distribution factor (PTDF) matrix \\
    $S \in \mathbb{R}^{m \times n}$ & indicator of generators' connection to nodes \\
    $f^{\max}\in \mathbb{R}^l$  & thermal capacity of transmission lines\\
    \bottomrule
    \end{tabular}%
  \label{tab:notations}
\end{table}

\section{Data-Driven IO Formulation}
\label{sec:Proposed_model}

Assume that multiple observations of schedules $\{x^0_i\}_{i \in \set{I}}$ and prices $\{\lambda^0_i,\omega^0_i\}_{i \in \set{I}}$ are given in the market-clearing results, where $\set{I}$ denotes the set of training data indexed by $i$. Using the data-driven IO formulation in \cref{mod:NIO_pred_multi}, the price recovery problem can be formulated based on primal variable $x$:
\begin{align}
   \min_{c,\{{x_i}\}_{i \in \set{I}}} \ \sum\nolimits_{i \in \set{I}} \left\|x_i^0 - {\rm{arg}} \min_{x_i} \ \textbf{DCOPF}(c)\right\|_p, \label{mod:NIO_form_1}
\end{align}
or based on dual variable $\lambda$:
\begin{align}
   \min_{c,\{{\lambda_i}\}_{i \in \set{I}}} \ \sum\nolimits_{i \in \set{I}} \left\|\lambda_i^0 - {\rm{arg}} \min_{\lambda_i} \ \textbf{DCOPF}(c)\right\|_p. \label{mod:NIO_form_2}
\end{align}
To bypass the complexity of bi-level problems, in this section we propose a data-driven IO formulation that incorporates the unknown cost vector $c$ into the objective function. We design the corresponding solution method based on gradient descent and prove its convergence to the global optimum. Additionally, we analyze the robustness of the data-driven IO in two different noisy settings, including (i) random errors in historical LMPs and (ii) mismatch between DCOPF and the actual market-clearing model.

\subsection{Data-Driven IO in Deterministic Settings}
\label{subsec:IO_no_errors}

We start from the single-observation case in a deterministic setting, where the observed solutions $\lambda^0$, $\omega^0$ and $x^0$ are assumed to be the optimal solutions of \cref{mod:DCOPF} and therefore satisfy the KKT conditions in \cref{mod:matrix_KKT}. From \cref{matrix_kkt:pg}, we know:
\begin{equation}
   c = \lambda J_{n,1} -\alpha +\beta + S^T \Phi^T(\nu -\mu).  \label{eq:c_expression}
\end{equation}
From \cref{eq:lmp_without_c}, we know:
\begin{equation}
  \Phi^T (\nu-\mu) = \omega - J_{m,1} \lambda . 
\label{eq:mu_nu_phi}
\end{equation}
Substituting \cref{eq:mu_nu_phi} into \cref{eq:c_expression} returns:
\begin{equation}
   c = \lambda (J_{n,1} - S^T J_{m,1}) + S^T \omega -\alpha +\beta. \label{eq:c_expression_new}
\end{equation}
which is a closed-form relationship between the unknown parameters $c$ and the observed solutions $\lambda^0$, $\omega^0$ and $x^0$. 
Thus, \cref{eq:c_expression_new} can be used to compute $c$ if the true values of $\alpha$ and $\beta$ are known. Normally, the values of dual multipliers $\alpha$, $\beta$, $\mu$, and $\nu$ are not given directly in the market-clearing results. Nonetheless, the binding status of constraints \cref{DCOPF:pg_max,DCOPF:pg_min,DCOPF:flow_max,DCOPF:flow_min} can be determined based on the published power system schedules; that is, if a constraint is non-binding, its corresponding dual multiplier is zero. Therefore, by only considering free generators whose capacity constraints \cref{DCOPF:pg_min} and \cref{DCOPF:pg_max} are not binding (i.e., $\alpha$ and $\beta$ are zero), the offer prices of these generators can be calculated directly by:
\begin{equation}
   c^0 = \lambda^0 (J_{n,1} - S^T J_{m,1}) + S^T \omega^0. \label{eq:c_deterministic}
\end{equation}
For generators with binding power constraints, their offer prices cannot be calculated by \cref{eq:c_deterministic}, so their corresponding values in vector $c^0$ are set to zero. 
By considering more scenarios with different free generators, the offer prices of more generators can be recovered. 
Note that although the generation schedule $x^0$ is not explicitly included in \cref{eq:c_deterministic}, it plays a crucial role in identifying the free generators. It is also worth noting that no information about the power transmission system, including the PTDF matrix $\Phi$ and the thermal capacity of transmission lines $f^{\max}$, is required in \cref{eq:c_deterministic}, which means that we can recover the offer prices of generators without assuming full knowledge of system parameters. If certain system parameters are unknown, we can recover them by applying another IO model after we have recovered the offer prices $c$ following this paper.

Then, we extend \cref{eq:c_deterministic} to the multi-observation case in a deterministic setting. Recall that the offer price of a generator may vary at different times due to external factors. In this case, while the observed solutions still satisfy the KKT conditions in \cref{mod:matrix_KKT}, calculating $c^0$ based on different sets of $\lambda^0$ and $\omega^0$ would yield different offer prices for the same generator. To address this, we propose the following data-driven IO formulation. Given multiple observed solutions $\{\lambda^0_i,\omega^0_i\}_{i \in \set{I}}$ where $\set{I}$ is the set of training data indexed by $i$, the objective of data-driven IO is:
\begin{equation}
  \min_{\hat c} \ l(\hat c)= \sum\nolimits_{i \in \set{I}} \| c_{i}^{0} - \hat c \|_p^p. \label{eq:data_driven_objective}
\end{equation}
where $c_{i}^{0}$ is calculated by substituting $\lambda^0_i$ and $\omega^0_i$ into \cref{eq:c_deterministic}. The objective of this IO is defined using the $p^\text{th}$-power of the $\ell_p$-norm, which is different from the $\ell_p$-norm used in other IO formulations discussed in Section~\ref{sec:IO_review}. This choice will aid in proving the convergence and uniqueness of the global optimum in the following subsections by enforcing strict convexity.

Note that the data-driven IO in \cref{eq:data_driven_objective} is a single-level problem, so it has lower computational complexity compared to the bi-level data-driven IO in \cref{mod:NIO_form_1,mod:NIO_form_2}. This advantage comes from the availability of a closed-form relationship between the unknown parameter $c^0$ and the observed solutions $\lambda^0$ and $\omega^0$ as specified in \cref{eq:c_deterministic}. Although this relationship is insightful in the deterministic setting, it does not hold in some noisy settings, which will be discussed in section~\ref{subsec:Robustness}.

Since only the offer prices of free generators can be recovered from (21), $c_i^0$ is a sparse vector. The zero elements in vector $c_i^0$ do not provide any additional information, so they should be eliminated from the loss function in (22) as:
\begin{equation}
  \min_{\hat c_k} \ l(\hat c_k)= \sum\nolimits_{i \in \set{I}_k} \| c_{i,k}^{0} - \hat c_k \|_p^p, \label{eq:data_driven_objective_modified}
\end{equation}
where $\hat c_k$ is the recovered price of generator $k$, and $\{c_{i,k}^{0}\}_{i \in \set{I}_k}$ is the valid (i.e., non-zero) training data for generator $k$. In this case, the zero elements in vector $c_i^0$ will not contribute to the loss function.

We can use the gradient descent (GD) method to iteratively update the values of $\hat c$. The update rule for $\hat c_k$ is:
\begin{align}
  \hat c_{k}^{j+1} = \hat c_{k}^{j} - \frac{\eta}{N_k} \sum\nolimits_{i=1}^{N_k} \frac{\partial \|c_{i,k}^0 - \hat c_{k}^{j} \|_p^p}{\partial \hat c_{k}^{j}},\ \forall k \label{eq:update_rule}
\end{align}
where $\hat c_{k}^{j}$ is the estimated offer price of generator $k$ during the $j^{\rm{th}}$ iteration, $\eta$ is the learning rate, $N_k$ is the amount of valid training data for generator $k$. Specifically, if we use $\ell_2$-norm in the loss function, i.e., $p=2$, there is a closed-form solution for $\hat c_k$, which is the mean of all the training data $\{c_{i,k}^{0}\}_{i \in \set{I}_k}$. 

Algorithm~\ref{alg:proposed} outlines the procure of solving the proposed data-driven IO model based on GD.
To enhance computational performance, stochastic gradient descent (SGD) can be used when dealing with large training data sets or when the price recovery task becomes an online learning problem.

\begin{algorithm}[t]
    \small
    \SetAlgoLined
    \SetKwInOut{Input}{input}\SetKwInOut{Output}{output}
    \Input{Observed market-clearing results $\{\lambda^0_i,\omega^0_i,x^0_i\}_{i \in \set{I}}$;\\
    Generator parameters $S$, $x^{\max}$, $x^{\min}$;\\
    }
    \Output {Recovered offer price $\hat c_k$ for $\forall \ k \in [1,n]$;
    }
    \Begin{
      \For{$i \in \set{I}$}
      {Identify free generators based on $x^0_i$, $x^{\max}$ and $x^{\min}$; \\
      Following \cref{eq:c_deterministic}, use $\lambda^0_i$, $\omega^0_i$ and $S$ to calculate $c^0_i$; \\
      For non-free generators, change the corresponding elements in vector $c^0_i$ to 0;\\
      }
      Following \cref{eq:update_rule}, use GD to find the optimal value of $\hat c_k$ based on $\{c^0_{i,k}\}_{i \in \set{I}_k}$ for $\forall \ k \in [1,n]$\\
      \KwRet{$\hat c_{k},\ \forall k \in [1,n]$}\\
     }
    \caption{Data-Driven IO for Offer Price Recovery}
    \label{alg:proposed}
\end{algorithm}

\subsection{Convergence of Data-Driven IO Based on Gradient Descent (GD)}
\label{subsec:convergence}
In this subsection, we prove that the loss function in \cref{eq:data_driven_objective_modified} converges to its optimal value with a sub-linear rate using the gradient descent (GD) method. The established convergence theorem for GD on convex, Lipschitz continuous functions can be found in Appendix \ref{subsec:theorem_convergence}. We first discuss the convexity and Lipschitz continuity of $\left\| x \right\|_p^p$, which are essential to prove the convergence of the proposed data-driven IO formulation.

\begin{lemma}
\label{Lemma:convexity_and_continuity_of_norms}
(Convexity and Lipschitz continuity of $\left\| x \right\|_p^p$) Let
$f(x) = \left\| x \right\|_p^p = \sum\nolimits_{i=1}^n |x_i|^p$. Assuming that $p \ge 1$ and $|x_i| \in [0,\bar x]$ for $\forall i \in [1,n]$, then $f(x)$ is convex and Lipschitz continuous with respect to the 
$\ell_p$-norm, and the Lipschitz constant is $p \bar x^{p-1} n^{p-1/p}$.
\end{lemma}

The proof of Lemma~\ref{Lemma:convexity_and_continuity_of_norms} can be found in Appendix~\ref{subsec:lemma_4_2}. Based on Theorem~\ref{Theorem:GD_convergence} and Lemma~\ref{Lemma:convexity_and_continuity_of_norms}, we can propose the convergence theorem for the data-driven IO in Section~\ref{subsec:IO_no_errors}:
\begin{theorem}
\label{Theorem:IO_convergence}
(Convergence of data-driven IO) Consider the loss function $l(c)$ in \cref{eq:data_driven_objective_modified}, where $c \in \set{C} \subseteq \mathbb{R}^n$ and $\|\set{C}\| \le B$. Assume that the $\ell_p$-norm in $l(c)$ satisfies $p\ge 1$ and $|c_k| \in [0,\bar c]$ for $\forall k \in [1,n]$. Let $c^{*}$ be an optimal solution of $\min_{c \in \set{C}} l(c)$ and let $\hat c$ be an output of applying the GD algorithm on $l(c)$. For $\forall \epsilon >0$, to achieve $l(\hat c) - l(c^*) \le \epsilon$, it is sufficient to perform GD for $T$ iterations with a learning rate of $\eta$, where $T \ge \frac{n^2 p^2 \bar c^{2p}}{\epsilon^2}$ and $\eta = (p \bar c^{p-2} n^{p-2/p} \sqrt{T})^{-1} $.
\end{theorem}

\begin{proof}
Since $|c_k| \le \bar c$ for $\forall k$, we have $B = (\sum\nolimits_{i=1}^n \bar c^p) ^{1/p} = \bar c n^{1/p}$.
Based on Lemma~\ref{Lemma:convexity_and_continuity_of_norms}, we can prove $l(c)$ is convex and $\rho$-Lipschitz continuous and the Lipschitz constant is $\rho = p \bar c^{p-1} n^{p-1/p}$. Therefore, $B^2 \rho^2 =  n^2 p^2 \bar c^{2p}$. According to Theorem~\ref{Theorem:GD_convergence}, we know that:
\begin{equation}
    l(\hat c) - l(c^{*}) \le \frac{B \rho}{\sqrt{T}} =  \frac{n p \bar c^{p}}{\sqrt{T}}. \label{eq:IO_convergence}
\end{equation}
Therefore, $T \ge \frac{n^2 p^2 \bar c^{2p}}{\epsilon^2}$ can ensure $\epsilon \le n p \bar c^{p} /\sqrt{T}$.
\end{proof}

Specifically, if $\ell_1$-norm is used in $l(c)$, i.e., when $p=1$, the error bound in \cref{eq:IO_convergence} becomes $\epsilon \le n \bar c /\sqrt{T}$. If the loss function is strongly convex, for example when $\ell_2$-norm is used, the convergence rate could be greater \cite{rakhlin2011making}. To satisfy additional limits on the value of $\hat{c}_k$, inequality constraints can be introduced into the unconstrained optimization model in \cref{eq:data_driven_objective_modified}. The resulting problem can then be solved using projected gradient descent. However, for the convergence theorem to hold, it is crucial that the feasible set of $\hat{c}_k$, as defined by the inequality constraints, be both closed and convex \cite{vu2022asymptotic}.

\subsection{Existence and Uniqueness of Global Optimal Solution to Data-Driven IO}
\label{subsec:existence_and_uniqueness}
This subsection focuses on demonstrating the existence and uniqueness of the global optimal solution to the data-driven IO formulation in \cref{eq:data_driven_objective_modified}. The established theorems for the existence and uniqueness of optimal solutions to convex problems can be found in Appendix \ref{subsec:theorem_Uniqueness}. First we prove the strict convexity of $\left\| x \right\|_p^p$, which is essential to prove the uniqueness of global optimum to the proposed data-driven IO formulation.

\begin{lemma}
\label{Lemma:strict_convexity_of_norm}
(Strict convexity of $\left\| x \right\|_p^p$) Following the notations in Lemma~\ref{Lemma:convexity_and_continuity_of_norms}, if $p >1$ and $x_i \ge 0$ for $\forall i \in [1,n]$, then $f(x) = \left\| x \right\|_p^p$ is strictly convex.
\end{lemma}

The proof of Lemma~\ref{Lemma:strict_convexity_of_norm} can be found in Appendix~\ref{subsec:lemma_4_6}. Based on Theorems~\ref{Theorem:exist_optimal}-- \ref{Theorem:unique_optimal} and Lemma~\ref{Lemma:strict_convexity_of_norm}, we can propose the following theorem for the data-driven IO formulation in \cref{eq:data_driven_objective_modified}:

\begin{theorem}
\label{Theorem:IO_existence_and_uniqueness}
(Existence and uniqueness of global optimum to data-driven IO) Following the notations in Theorem~\ref{Theorem:IO_convergence}, assuming the $\ell_p$-norm in $l(c)$ satisfies $p > 1$ and $c_k \in [\underline{c},\overline{c}]$ for $\forall k \in [1,n]$ where $\underline{c}>0$, then $\min_{c \in \set{C}} l(c)$ has a unique global optimal solution $c^{*}$, and $\hat c \to c^{*}$ when the iteration $T \to \infty$.
\end{theorem}
\begin{proof}
According to Theorem~ \ref{Theorem:unique_optimal} and Lemma~\ref{Lemma:strict_convexity_of_norm}, we know that $l(c)$ is strictly convex and thus has at most one optimal solution in the convex set $\set{C}$. Meanwhile, since the feasible region of $c$ is closed and bounded and $l(c)$ is continuous on $\set{C}$ (Lipschitz continuity is stronger than uniform continuity), we know that $\min_{c \in \set{C}} l(c)$ has at least one global optimum according to Theorem~\ref{Theorem:exist_optimal}. In summary, $\min_{c \in \set{C}} l(c)$ has a unique global optimal solution $c^*$. Given the error bound in Theorem~\ref{Theorem:IO_convergence}, we know that $l(\hat c) - l(c^{*}) \to 0$ when the iteration $T \to \infty$. Considering the uniqueness of $c^*$, we can prove that $\hat c \to c^{*}$ when $T \to \infty$.
\end{proof}

Inspecting Theorems~\ref{Theorem:IO_convergence} and \ref{Theorem:IO_existence_and_uniqueness}, we notice two differences on the assumptions of $l(c)$: Theorem~\ref{Theorem:IO_convergence} assumes $p \ge 1$ and $0\le |c_k|\le \bar c$ for $\forall k$, while Theorem~\ref{Theorem:IO_existence_and_uniqueness} assumes $p > 1$ and $\underline{c} \le c_k\le \overline{c}$ for $\forall k$. Variable $c_k$ represents the offer price of generator $k$, which is greater than its generation cost and capped due to market power regulations. Hence, the assumption of $\underline{c} \le c_k\le \overline{c}$ is realistic. We remark that in some exceptional cases, under heavy production credits, renewable generators can submit negative offer prices; however, this practice is expected to gradually phase out. 

Theorems~\ref{Theorem:IO_convergence} and \ref{Theorem:IO_existence_and_uniqueness} only apply to generators that have been marginal at least once in the training data set. The offer prices of non-marginal generators are eliminated from the loss function $l(c)$ in \cref{eq:data_driven_objective_modified}, so the estimations of these prices remain unchanged from their initial values, rather than converging to an optimal value.

\subsection{Robustness of Data-Driven IO in Noisy Settings}
\label{subsec:Robustness}

\subsubsection{Noisy Setting I: Random Errors in LMPs}
\label{subsec:type_1_errors}

The market-clearing results published by ISOs may not reflect the exact optimal solutions due to incorrect inputs or calculation errors. However, these results are generally \textit{believed} to be accurate and are \textit{accepted} by market participants due to the price validation procedure conducted after the market is closed \cite{LMP_slides}. According to \cref{eq:c_deterministic}, the data-driven IO formulation will not be affected by errors in the observations of primal variable $x$ due to its independence from $x^0$. However, errors in LMPs can compromise the expression of $\omega$ in \cref{eq:lmp_with_c} and therefore the expression of $c^0$ in \cref{eq:c_deterministic}. This kind of noise in data-driven IO is similar to the estimation error in statistical ML which is implied by the fact that the ML model is based on a finite training data set that only partially reflects the true distribution of data. The GIO and NIO formulations in Section~\ref{sec:IO_review} can handle random errors in training data. However, we choose to retain the original form in \cref{eq:data_driven_objective_modified} instead of adapting it to a GIO or NIO formulation due to the sparsity of errors. According to the parameter update rule of GD in \cref{eq:update_rule}, the impact of a single training data point on the final result is small when the training set is large. Meanwhile, the $\ell_1$-norm is robust to outliers. In summary, we can reduce the impact of random errors by using $\ell_1$-norm in the loss function or increasing the amount of training data. We will explore these methods further in Section~\ref{subsec:small_system}.

\subsubsection{Noisy Setting II: Mismatch Between DCOPF and Actual Market-Clearing Model}
\label{subsec:type_2_errors}
As mentioned in Section~\ref{subsec:SCUC}, the actual day-ahead market-clearing is based on the SCUC model instead of the DCOPF model employed in this paper. This mismatch is the second source of noise in the price recovery problem and is comparable to the approximation error in statistical ML (i.e., the difference between the best model within the chosen model class and the optimal model within all model classes). To some extent, the proposed data-driven IO formulation is robust to this noise, as some additional constraints in SCUC do not compromise the validity of \cref{eq:c_deterministic} even though it is derived based on DCOPF. For instance, adding the flexibility reserve requirements to \cref{mod:DCOPF} results in the following DCOPF model:
\begin{subequations}
\begin{align}
   & \min_{x,r^{+},r^{-},\lambda,\alpha,\beta,\mu,\nu,\theta^{+},\theta^{-}}\ c^T x \label{DCOPF_reserve:obj}\\
\text{s.t. } \ 
   & \text{\cref{DCOPF:power_balance,DCOPF:flow_min,DCOPF:flow_max}} \nonumber\\
   &(\theta^{+}): J_{n,1}^T r^{+} \ge R^{+} \label{DCOPF_reserve:upward_reserve}\\
   &(\theta^{-}): J_{n,1}^T r^{-} \ge R^{-} \label{DCOPF_reserve:downward_reserve}\\
   &(\alpha): x+r^{+} \le  u^*x^{\max}. \label{DCOPF_reserve:pg_max} \\
   &(\beta): x-r^{-} \ge  u^*x^{\min}. \label{DCOPF_reserve:pg_min}
\end{align}%
\label{mod:DCOPF_reserve}%
\end{subequations}%
\allowdisplaybreaks[0]%
where $r^{+}$ and $r^{-}$ are the upward and downward flexible capacities provided by generators, and $R^{+}$ and $R^{-}$ are the upward and downward reserve requirements. In this case, the KKT condition in \cref{matrix_kkt:pg} and the LMP definition in \cref{eq:lmp_without_c} still hold, so the expression of $\lambda$ in \cref{eq:reference_price} and the expression of $\omega$ in \cref{eq:lmp_with_c} remain unchanged. This means that $c^0$ can still be calculated based on \cref{eq:c_deterministic} and the objective function in \cref{eq:data_driven_objective_modified} can still be used to recover the cost vector $c$. The only change would be the number of free generators, which is determined by the quantity of generators operating between their minimum and maximum output limits.
The power limits in \cref{DCOPF_reserve:pg_max,DCOPF_reserve:pg_min} are more restrictive compared to \cref{DCOPF:pg_max,DCOPF:pg_min}, resulting in fewer free generators and fewer recoverable offer prices from a single observation. However, these additional constraints also create more diverse scenarios where different generators are marginal, thus potentially increasing the total number of recoverable offer prices from historical data. Note that some additional constraints may render \cref{eq:c_deterministic} invalid, such as the ramping limits of generators:
\begin{align}
    & (\varphi_t^{+}):  x_{t} +r_{t}^{+} - x_{t-1} +r_{t-1}^{-}\leq R^{up} u_{t-1} + v_{t} x^{\min} , \ \forall t\in\set{T} \label{eq:ramp_up}\\
    & (\varphi_t^{-}):  x_{t-1}+r_{t-1}^{+} - x_{t} +r_{t}^{-} \leq R^{down} u_{t} + w_{t} x^{\min} , \ \forall t\in\set{T}\label{eq:ramp_down}
\end{align}
where $R^{up}$ and $R^{down}$ are the ramp-up and down rates of generators, $v_{t}$ and $w_{t}$ are binary variables denoting the start-up and shut-down decisions at time $t$. In this case, the dual multipliers $\varphi^{+}_t$ and $\varphi^{-}_t$ will be added to the expression of $c$, and their influence on the IO results is determined by the frequency of \cref{eq:ramp_up,eq:ramp_down} being binding. 

The model mismatch problem is a common challenge for all IO formulations and cannot be fully resolved without using a more accurate FO model. We will study the impact of this mismatch on IO results using numerical experiments in Section~\ref{subsec:large_system} and discuss other possible solutions in Section~\ref{sec:conclusion}.

\section{Case Study}
\label{sec:Case_study}

In this section we demonstrate the effectiveness of the proposed data-driven IO based on the IEEE 14-bus system and the NYISO 1814-bus system. Simulations were conducted using Python v3.8 and the Gurobi solver on a standard PC with an Intel i9 processor and 16 GB of RAM. In the IEEE 14-bus system, the data-driven IO model was solved in one minute. With regard to the NYISO system, the solution of the proposed model was obtained in all instances in less than ten minutes.

\subsection{Illustrative Example: IEEE 14-bus system}
\label{subsec:small_system}

We first evaluate the performance of the data-driven IO in the IEEE 14-bus system (details in Appendix~\ref{subsec:IEEE_system}). We assume that each generator sets a piece-wise linear price baseline with five equally divided blocks based on the quadratic power generation cost, as shown in Fig.~\ref{fig:offer_prices}. The actual offer prices fluctuate around the baseline with deviations following a normal distribution $N(\mu,\sigma^2)$, where $\mu = 0$ and $\sigma = 2$ (unit: \$/MWh) are assumed in this case.

\begin{figure}[!t]
    \centering
    \includegraphics[width=0.9\linewidth]{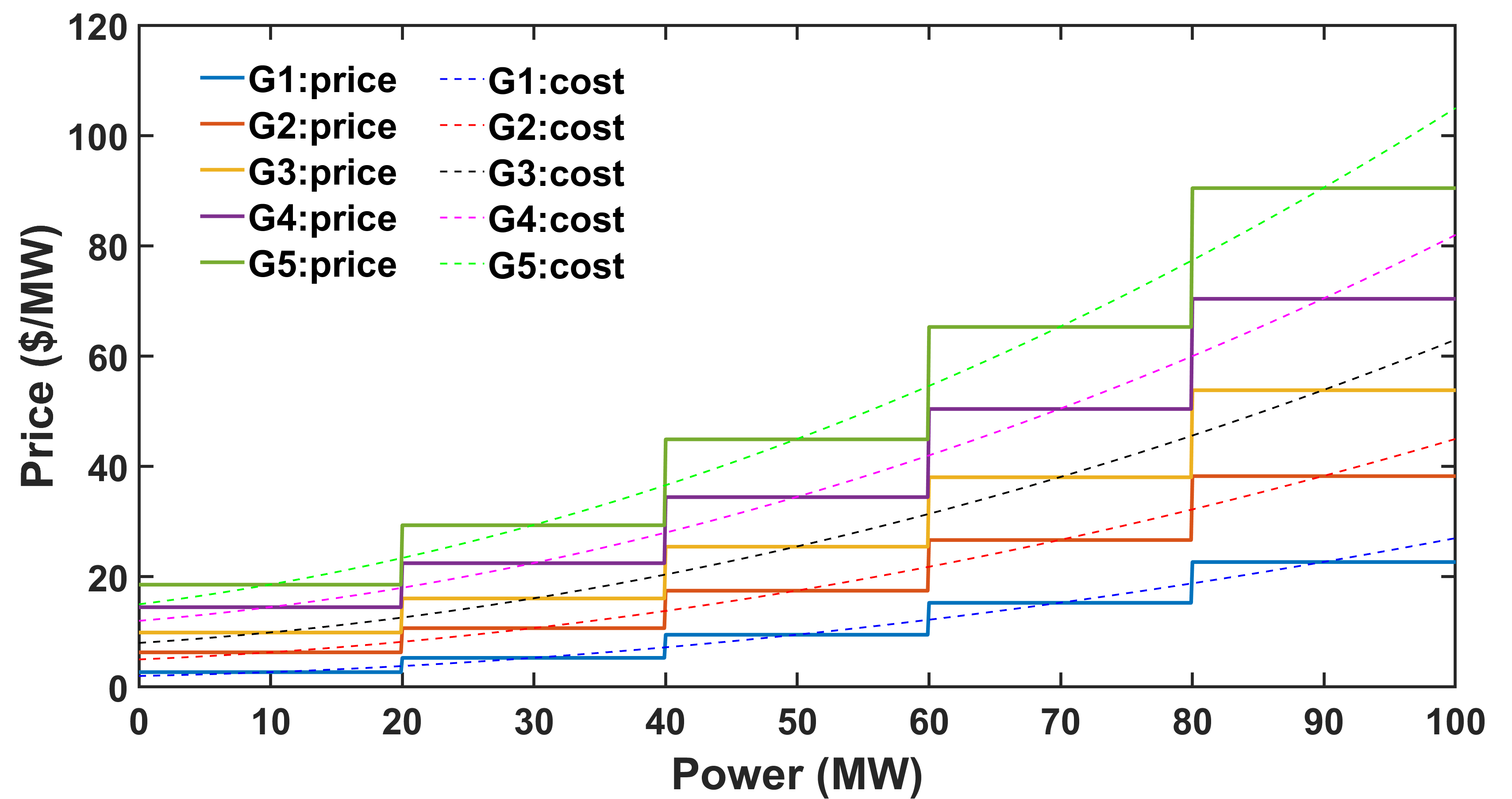}
    \caption{Power generation costs (dashed lines) and price baselines (solid lines) of five generators.}
    \label{fig:offer_prices}
\end{figure}

\subsubsection{Performance in Deterministic Settings}
\label{subsubsec:no_errors}

To recover as many offer prices as possible, we create 200 training data points and validate that each of the five generators is marginal at least once at each block. The initial values of the five blocks are set to 10, 20, 30, 40, and 50, respectively. Fig.~\ref{fig:convergence} shows the variations of four estimated offer prices at each iteration and compares the convergence rates based on the $\ell_1$ and $\ell_2$-norms. In all four cases, the estimated prices converge to the true values, with faster convergence observed when using the $\ell_2$-norm due to its strong convexity.

\begin{figure}[!t]
    \centering
    \includegraphics[width=1\linewidth]{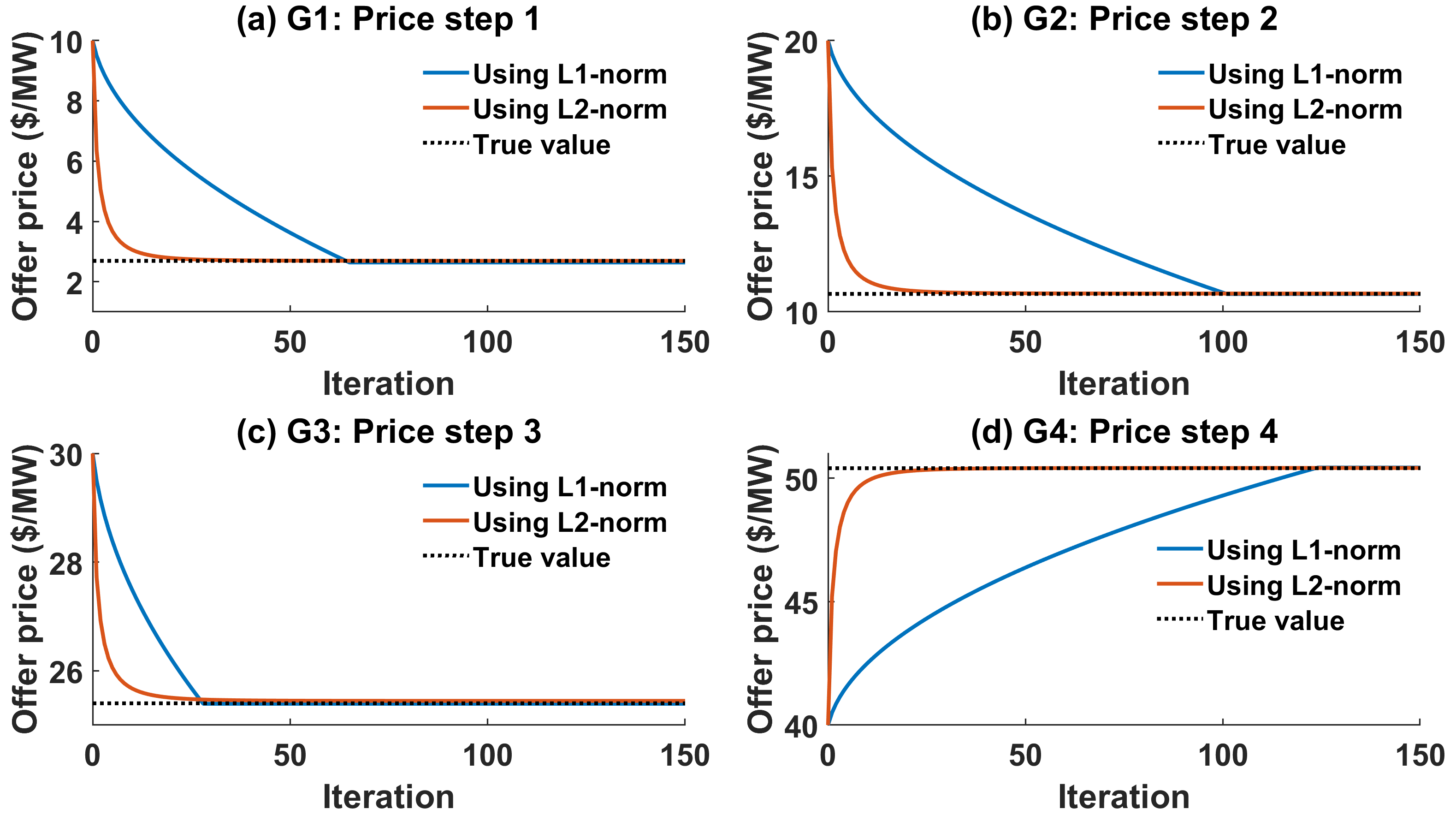}
    \caption{Convergence of estimated offer prices based on $\ell_1$ and $\ell_2$-norms.}
    \label{fig:convergence}
\end{figure}

\subsubsection{Performance with Random Errors in LMPs}
\label{subsubsec:type_1_errors}
To evaluate the performance of data-driven IO in the first noisy setting, we add random errors following a normal distribution $N(\mu,\sigma^2)$ to the original LMPs. Four noise settings are considered with varying error sizes and frequencies. In the small-error case, we assume $\mu = 50$ and $\sigma = 5$, while in the large-error case, we assume $\mu = 100$ and $\sigma = 10$ (unit: \$/MWh). According to \cite{LMP_slides}, the actual frequency of random errors in LMPs is less than 0.1\%, but we consider more challenging cases with 1\% and 5\% error frequencies. We use data-driven IO to recover the offer price of G3 whose true value is 25.445 \$/MWh, and the estimated offer prices and the corresponding errors are shown in Table~\ref{tab:error_1}. The data-driven IO shows robustness to random errors with $\ell_1$-norm and acceptable performance with $\ell_2$-norm when noise is small and sparse.

\begin{table}[!t]
  \centering
  \caption{Estimated offer prices ($\hat c$) and corresponding errors ($\epsilon$) in different noisy settings (true value $c^*= 25.445$ \$/MWh)}
    \begin{tabular}{ccc}
    \toprule
    Noisy settings	& Using $\ell_1$-norm &  Using $\ell_2$-norm	\\
    \midrule
    1\% small errors  & $\hat c =25.56$, $\epsilon =0.45\%$  & $\hat c =25.69$, $\epsilon =0.96\%$ \\
    1\% large errors  & $\hat c =25.56$, $\epsilon =0.45\%$  & $\hat c =26.19$, $\epsilon =2.93\%$ \\
    5\% small errors  & $\hat c =25.56$, $\epsilon =0.45\%$  & $\hat c =26.67$, $\epsilon =4.81\%$ \\
    5\% large errors  & $\hat c =25.56$, $\epsilon =0.45\%$  & $\hat c =29.17$, $\epsilon =14.64\%$ \\
    \bottomrule
    \end{tabular}%
  \label{tab:error_1}
\end{table}

\subsection{Numerical Experiments on NYISO System}
\label{subsec:large_system}

In this numerical experiment, we study the NYISO system with 1814 buses, 2207 lines, 362 generators and 33 wind farms (details in  Appendix~\ref{subsec:NYISO_system}). The structure and parameters of this NYISO system are mined and estimated from publicly available data sources \cite{NYISO_RNA_report}. The NYISO publishes daily LMPs at \cite{NYISO_price} and other market-clearing results every three months at \cite{NYISO_market_clearing_results}. We use this information to generate hourly market-clearing results over a three-month period (February, April, and August 2018), yielding a total of 2136 training data points. Note that in this case we assume that the baseline price of each generator is equally divided into ten price blocks, and the upper and lower bounds of each block are known. However, if the block setting rules are unpublished, we need to estimate the bounds of each block first based on historical data. Subsequently, we can recover the offer price of each block based on the estimated bounds of the blocks.

\subsubsection{Performance in Deterministic Settings}
As previously mentioned in Section~\ref{subsec:IO_no_errors}, a training data point is valid for the price recovery of a generator if the generator is committed and marginal in that scenario. Fig.~\ref{fig:step_stat} summarizes the number of valid training data points for all the generators in the system. As shown in Fig.~\ref{fig:step_stat}, most of the generators only have valid training data at a few steps of their offer prices, thus only a portion of the prices can be recovered. The results indicate that 44.03\% of all offer prices can be recovered from the valid training data, while 81.21\% of these recovered prices are based on less than 5 training data points. Meanwhile, for 85.35\% of the generators, at least one block of price can be recovered. Note that there is an theoretical upper limit on the proportion of recoverable prices because an offer price can only be recovered from the LMPs if it affected the values of LMPs. Therefore, other approaches, such as NIO, also face the same limitation with regards to the recovery rate.

In total, 53 generators (out of 362) are never found to be marginal in the 2163 data points we analyzed, so their offer prices cannot be estimated from the training data.
This result does not indicate a failure of the data-driven IO for large systems, as these non-marginal generators do not typically compete with others. For example, due to a relatively high generation cost, generator No.14 is never committed in this three months. Furthermore, incorporating more training data can decrease the number of unrecoverable generators. As more market-clearing data posted by market operators, the portion of recovered offer prices will increase, which is important for dealing with noisy data. 

\begin{figure}[!t]
    \centering
    \includegraphics[width=1\linewidth]{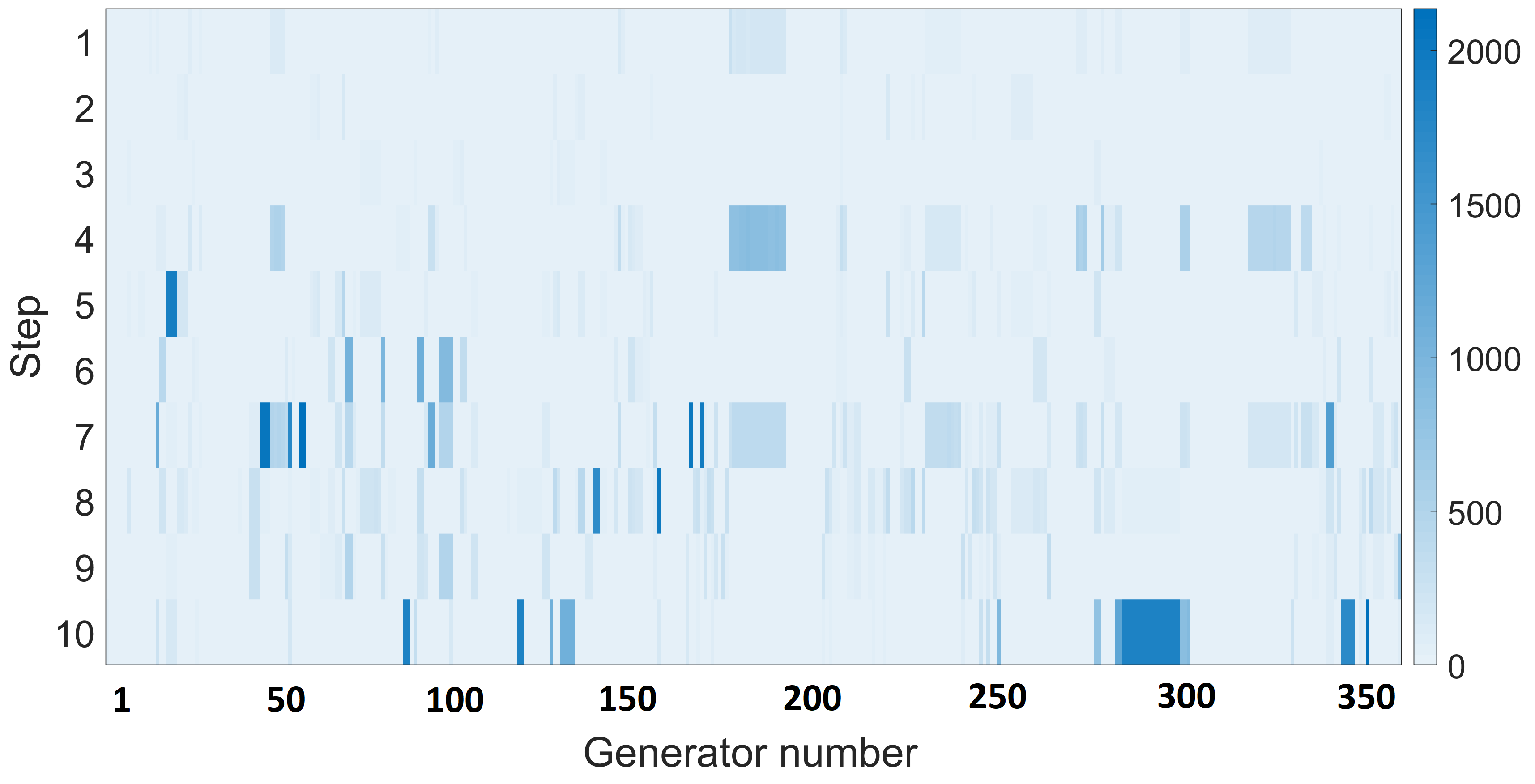}
    \caption{Number of valid training data points for 362 generators.}
    \label{fig:step_stat}
\end{figure}

\subsubsection{Performance with Model Mismatch}
\label{subsubsec:type_2_errors}
To assess the performance of the data-driven IO approach in the presence of model mismatches, we replicate the SCUC model used by NYISO in its day-ahead scheduling process, as outlined in \cite{DA_scheduling_manual}. This SCUC model, as detailed in \cite{mieth2022risk}, includes additional constraints such as ramping limits for generators, which may challenge the accuracy of recovered prices when those constraints are binding.
Fig.~\ref{fig:binding_constraints} shows the binding frequency of two sets of constraints, including the power limits in \cref{DCOPF_reserve:pg_max,DCOPF_reserve:pg_min} and the ramping limits in \cref{eq:ramp_up,eq:ramp_down}. It also compares the results across three months, namely February (with low demand), April (with medium demand), and August (with high demand). As depicted in Fig.~\ref{fig:binding_constraints} (b), the occurrence of only the ramping limits being binding is rare. As a result, the frequency of both sets of constraints being binding is approximately the same as the frequency of only the power limits being binding.

Fig.~\ref{fig:binding_constraints} conveys a crucial message that the ramping limits of free generators (whose power limits are not binding) are typically not binding either. As explained in Section~\ref{subsubsec:type_2_errors}, the ramping limits will only affect the recovered prices when they are binding. Hence, the impact of ramping limits is mostly eliminated from the valid training data set, which only includes the data of free generators. 

With this model mismatch, we are able to recover 36.28\% of all offer prices with an average relative error of 3.47\%. Fig.~\ref{fig:errors} shows the errors in recovered prices at the seventh block where 171 generators are marginal. The errors are compared in the deterministic setting and two noisy settings, where noisy setting I includes additional Gaussian noise to 1\% LMPs data with a mean value of 50 \$/MWh and a standard deviation of 5 \$/MWh, and noise setting II involves the model mismatch discussed above.
The deterministic setting features small and random errors with an average close to zero, while in noisy setting I, sparse noise in the training data results in significant deviations of the recovered prices of some generators from their true values. Noisy setting II features higher errors compared to the other two cases, however, they are still within an acceptable level compared to the offer prices. Note that although the errors in Fig.~\ref{fig:errors} (c) tend to be mostly positive, it is not a universal conclusion since the values of errors in this noisy setting depend on the values of dual multipliers of binding constraints in the SCUC model, which can change over time.

\begin{figure}[!t]
    \centering
    \includegraphics[width=1\linewidth]{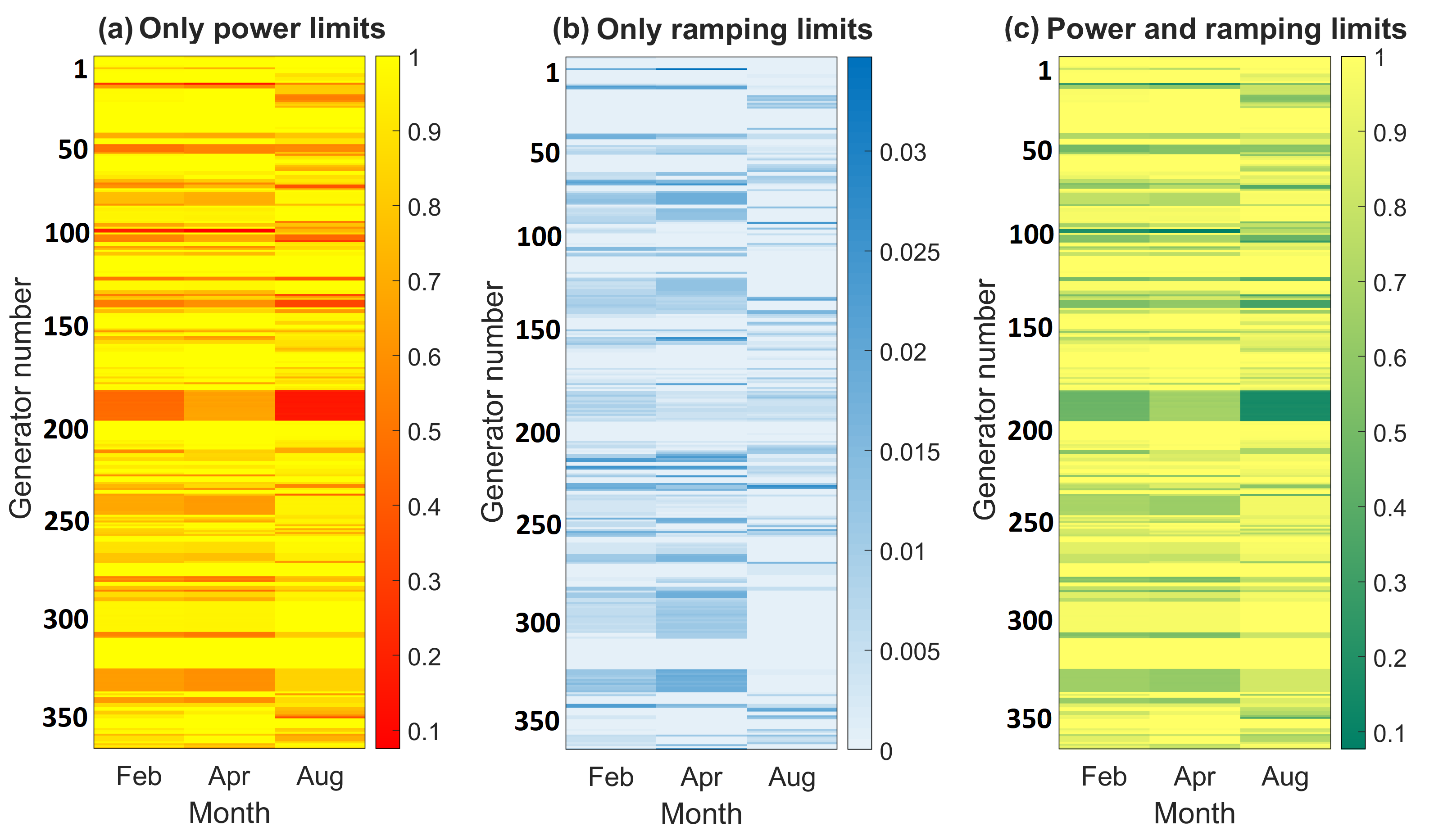}
    \caption{Binding frequency of different sets of constraints.}
    \label{fig:binding_constraints}
\end{figure}

\begin{figure}[!t]
    \centering
    \includegraphics[width=1\linewidth]{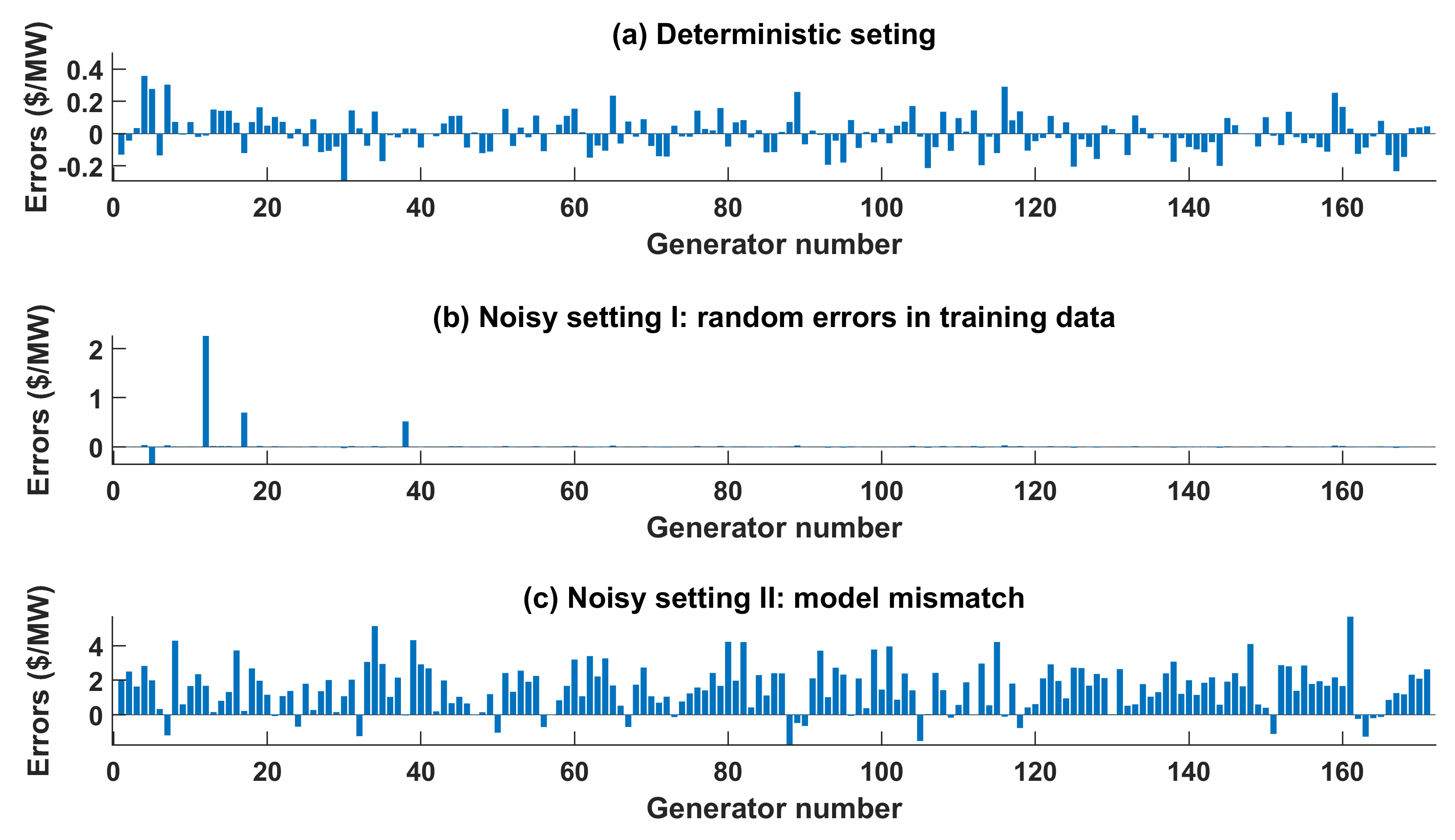}
    \caption{Errors in recovered prices in deterministic and two noisy settings.}
    \label{fig:errors}
\end{figure}

\section{Conclusion and Future Work}
\label{sec:conclusion}
This paper presents a data-driven IO approach for recovering marginal offer prices of generators from day-ahead market-clearing results. We formulate the data-driven IO problem as a single-level optimization model, distinct from the bi-level models that are commonly used in the literature. This data-driven IO problem can be efficiently solved using the gradient descent method. Furthermore, We prove that the recovered offer prices converge to a unique global optimum with a sub-linear rate. Numerical experiments on the IEEE 14-bus system and NYISO 1814-bus system validate the efficacy of the proposed approach in both deterministic and noisy settings. The 14-bus test case shows that the use of $\ell_2$-norm in the loss function results in faster convergence, while the $\ell_1$-norm offers robustness against sparse noise in the training data. In the 1814-bus test case, the data-driven IO recovers approximately 45\% of all offer prices using 2163 training data points. The average relative error is 0.88\% in deterministic setting and 3.47\% considering the mismatch between the DCOPF model and the actual market-clearing process.


A notable limitation of our data-driven IO model is that it assumes complete knowledge of all system parameters except offer prices of generators, which may not be achievable in \textit{some} real-world power systems. Nevertheless, the presence of wholesale markets and stringent regulatory requirements enhance the likelihood of market transparency.
Furthermore, while ISO manuals (such as \cite{DA_scheduling_manual}) provide insights into the market-clearing principles, they do not guarantee complete visibility of the constraints used in the market-clearing process, including out-of-optimization interventions conducted by system operators. As part of our ongoing research, we plan to tackle this issue by implementing a machine learning approach, such as a physics-informed neural network, that can incorporate known constraints while also learning to identify previously unknown constraints from training data. This will significantly reduce the inaccuracies in the recovered prices resulting from model mismatches.

The proposed method for price recovery can lead to increased profits for certain market participants through strategic bidding, but it must not undermine the fairness of power markets. In the event that the price recovery method becomes widespread, the market operator must implement effective measures to prevent market power abuse. Therefore, our future study will concentrate on determining the most efficient policy for market power mitigation and investigating the optimal bidding strategies for market participants. In summary, the game between market participants and market operators will become more intricate and engaging due to the recovery of offer prices.

Furthermore, as depicted in Fig.~\ref{fig:power_market}, the energy market comprises a day-ahead market (DAM) and a real-time market (RTM). This paper only studies the price recovery problem in DAM, while the problem is more challenging in RTM due to the spontaneous actions of market participants and the ad-hoc manipulations of market operators. Uncovering the behavior patterns of market participants in RTM holds potential for both market participants seeking to increase their profits and market operators responsible for preventing market power abuse. In the next step, we intend to expand our data-driven IO approach to RTM and ancillary service markets to leverage all available market data.

\begin{acks}
The authors express their gratitude to Prof. Daniel Bienstock and Dr. Robert Mieth for their assistance in developing the digital twin of the NYISO system, data collection and visualization.
\end{acks}

\bibliographystyle{ACM-Reference-Format}
\bibliography{literature}

\appendix
\section{Established Theorems}

\subsection{Theorem for the Convergence of GD}
\label{subsec:theorem_convergence}
\begin{theorem}
\label{Theorem:GD_convergence} 
   (Convergence of GD for convex and Lipschitz continuous functions) Let $f$ be a convex, $\rho$-Lipschitz function and let $w^* \in {\rm{arg}} \min_{w \in \set{W}} f(w)$, where $\|\set{W}\| \le B$. If the GD algorithm is applied on $f$ for $T$ iterations with a learning rate of $\eta = \sqrt{\frac{B^2}{\rho^2 T}}$, then the output vector $\hat w$ satisfies:
\begin{equation}
   f(\hat w) - f(w^*) \le \frac{B \rho}{\sqrt{T}}. \label{eq:GD_convergence}
\end{equation}
\end{theorem}
Since this $\mathcal{O}(1/\sqrt T)$ convergence rate of GD is well-established in the literature, e.g., in \cite{shalev2014understanding}, we omit the proof of Theorem~\ref{Theorem:GD_convergence}. 

\subsection{Theorems for Existence and Uniqueness of Optimal Solutions}
\label{subsec:theorem_Uniqueness}
\begin{theorem}
\label{Theorem:exist_optimal}
(Existence of global optimum)
If the objective function is continuous and the feasible region is closed and bounded, then there exists a global optimum.
\end{theorem}

\begin{theorem}
\label{Theorem:unique_optimal}
(Uniqueness of optimal solutions to strictly convex functions) If the objective function is strictly convex and the feasible region is convex, then there exists at most one optimal solution.
\end{theorem}

We omit the proof of Theorems~\ref{Theorem:exist_optimal} and \ref{Theorem:unique_optimal} in this paper since they are rigorously proved in \cite{boyd2004convex}.

\section{Proofs}


\subsection{Proof of Lemma 4.2}
\label{subsec:lemma_4_2}
We first clarify the definition of Lipschitz continuity and prove three lemmas relevant to the proof of Lemma 4.2.
\begin{definition}
\label{Definition:Lipschitz_functions}
   ($\rho$-Lipschitz functions) Let $\set{W} \subseteq \mathbb{R}^n$ be a convex set. A function $f:\mathbb{R}^n \to \mathbb{R}$ is $\rho$-Lipschitz over $\set{W}$ if there exists a constant $\rho > 0$ that for $\forall w_1, w_2 \in \set{W}$, we have $| f(w_1) - f(w_2) | \le \rho \| w_1 - w_2 \|_p$. 
\end{definition}
Note that $\left\| \cdot \right\|_p$ in the definition of Lipschitz continuity is usually $\ell_2$-norm, but other norms are also applicable \cite{paulavivcius2006analysis}. The statements of "$f$ being $\rho$-Lipschitz" and "$f$ being Lipschitz continuous with Lipschitz constant $\rho$" are equivalent. 

\begin{lemma}
\label{Lemma:Lipschitz_of_norms}
(Lipschitz continuity of norms) Every norm on $\mathbb{R}^n$ is 1-Lipschitz with respect to the same kind of norm.
\end{lemma}
\begin{proof}
Let $f(x) = \left\| x \right\|_p = (\sum\nolimits_{i=1}^n |x_i|^p)^{1/p}$ be the $\ell_p$-norm of $x$. Based on the triangle inequality of norms (i.e., $ \| x+y\|_p \le \| x \|_p + \| y \|_p,\ \forall x,y \in\mathbb{R}^n$), we know that $\| x\|_p=\|x-y+y\|_p \le \|x-y\|_p+\| y\|_p$. Therefore, $\| x\|_p-\| y\|_p \le \|x-y\|_p$. Similarly we have $\| y\|_p =\|y-x+x\|_p \le \|y-x\|_p+\| x\|_p= \|x-y\|_p+\| x\|_p$. Therefore, $\| y\|_p-\| x\|_p\le \|x-y\|_p$. In summary, we have $\left| \| x\|_p-\| y\|_p \right|\le \|x-y\|_p$, i.e., $f$ is 1-Lipschitz with respect to the $\ell_p$-norm.
\end{proof}

\begin{lemma}
\label{Lemma:Lipschitz_and_differentiable}
(Lipschitz continuity of differentiable functions) An everywhere differentiable function $f: \mathbb{R} \to \mathbb{R}$ is Lipschitz continuous if it has bounded first-order derivative, and the corresponding Lipschitz constant is $\sup |f'(x)|$.
\end{lemma}
\begin{proof}
Let $f$ be a continuous function on the closed interval $[a,b]$ and differentiable on the open interval $(a,b)$. According to the mean value theorem, there exist a $\xi \in (x,y)$ such that $f(x)-f(y)=f'(\xi)(x-y)$. Taking the absolute values yields $|f(x)-f(y)|\le \rho |x-y|$, where $\rho$ is the supremum of $|f'(x)|$ over $(x,y)$.
\end{proof}

\begin{lemma}
\label{Lemma:Lipschitz_of_composition}
(Lipschitz continuity of compositions) The composition of a $\rho_1$-Lipschitz function and a $\rho_2$-Lipschitz function is a $\rho_1 \rho_2$-Lipschitz function.
\end{lemma}
\begin{proof}
Assume $g: \mathbb{R}^n \to \mathbb{R}$ is $\rho_1$-Lipschitz continuous and $h: \mathbb{R} \to \mathbb{R}$ is $\rho_2$-Lipschitz continuous. Let $f(x)= h(g(x))$. For $\forall a, b \in \mathbb{R}^n$, we have $|f(a), f(b)| = | h(g(a)), h(g(b)| \le \rho_2 | g(a), g(b))| \le \rho_2 \rho_1 \|a, b\|$. Therefore, $f$ is $\rho_1 \rho_2$-Lipschitz continuous.
\end{proof}

\noindent \textbf{Proof of Lemma 4.2:}
For convexity: According to \cite{boyd2004convex}, $|x|^p$ is convex on $\mathbb{R}$ for $\forall p \ge 1$, and the non-negative weighted sum of convex functions is convex. Therefore, $f$ is convex since it is the sum of convex functions.

For Lipschitz continuity: Let $g:\mathbb{R}^n \to \mathbb{R}$ be the $\ell_p$-norm and $h:\mathbb{R} \to \mathbb{R}$ be the $p^\text{th}$ power function, then $f(x) = h(g(x))$. According to Lemmas~\ref{Lemma:Lipschitz_of_norms}, $g$ is 1-Lipschitz continuous. Since $|x_i| \in [0,\bar x]$, we know that $g(x) = (\sum\nolimits_{i=1}^n |x_i|^p)^{1/p}$ is bounded by 0 and $(\sum\nolimits_{i=1}^n \bar x^p)^{1/p} = \bar x n^{1/p}$. Meanwhile, based on Lemma~\ref{Lemma:Lipschitz_and_differentiable}, it can be proved that $h(x)= x^p$ is Lipschitz continuous over interval $[0, \bar x n^{1/p}]$ for $\forall p \ge 1$, and the Lipschitz constant is $p \bar x^{p-1} n^{p-1/p}$. Therefore, based on Lemma~\ref{Lemma:Lipschitz_of_composition}, we know that $f$ is $\rho$-Lipschitz continuous, where $p \bar x^{p-1} n^{p-1/p}$. $\hfill\square$

\subsection{Proof of Lemma 4.6}
\label{subsec:lemma_4_6}
We first recall the definition of strict convexity and its first and second-order conditions:
\begin{definition}
\label{Definition:strictly_convex_functions}
  (Strictly convex functions) Let $\set{W} \subseteq \mathbb{R}^n$ be a convex set. If a function $f:\mathbb{R}^n \to \mathbb{R}$ satisfies $f(k w_1 +(1-k) w_2) < k f(w_1) +(1-k) f(w_2)$ for $\forall w_1, w_2 \in \set{W}$ and $k \in (0,1)$, it is strictly convex. The first-order condition for strict convexity is $f(w_2) > f(w_1) + \nabla f(w_1)^T (w_2 - w_1)$ for $\forall w_1, w_2 \in \set{W}$, and the second-order condition is $\nabla^2 f(w) \succ 0$ for $\forall w \in \set{W}$. 
\end{definition}

\noindent \textbf{Proof of Lemma 4.6:}
Since $f(x) = \sum\nolimits_{i=1}^n |x_i|^p=\sum\nolimits_{i=1}^n x_i^p$ when $x_i \ge 0$ for $\forall i$, we have $\partial f/\partial {x_i} = p x_i^{p-1}$ and $\partial^2 f/\partial {x_i}^2 = p (p-1) x_i^{p-2}>0$ for $\forall x_i >0$ and $\forall p > 1$. Since $\partial^2 f/{\partial {x_i} \partial {x_j}} = 0, \forall i \ne j$, $\nabla^2 f(x)$ is a diagonal matrix where the diagonal elements $\partial^2 f/\partial {x_i}^2, \forall i$ are the corresponding eigenvalues of $\nabla^2 f(x)$. Therefore, $\nabla^2 f(x)\succ 0$ since all of its eigenvalues are positive, which means that the second-order condition for strict convexity is satisfied. $\hfill\square$

\section{Data for Case Study}
\label{sec:data}
\subsection{Data of IEEE 14-bus system}
\label{subsec:IEEE_system}

The IEEE 14-bus system is depicted in Fig.~\ref{fig:14_bus} \cite{IEEE_14_bus} and the parameters of the five generators are listed in Table~\ref{tab:G_parameters}. The power generation cost of each generator follows a quadratic equation $c_0+c_1 x+c_2 x^2$, where $x$ is the output power of the generator. 

\begin{figure}[!t]
    \centering
    \includegraphics[width=0.9\linewidth]{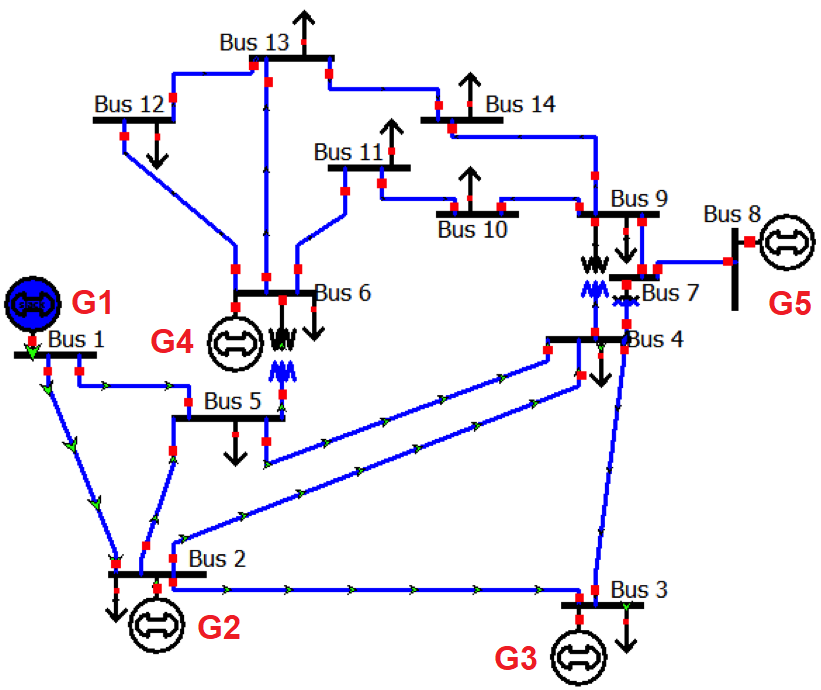}
    \caption{IEEE 14-bus system \cite{IEEE_14_bus}.}
    \label{fig:14_bus}
\end{figure}

\begin{table}[!t]
  \centering
  \caption{Parameters of generators in IEEE 14-bus system}
    \begin{tabular}{ccccc}
    \toprule
    {No.}  & $x^{\max}/x^{\min}$ (MW)  & $c_0$ ($\$$) & $c_1$ ($\$$/MWh)   & $c_2$ ($\$$/MWh$^2$) \\
    \midrule
    G1    & 100/0   & 2    & 0.05     & 0.002 \\
    G2    & 100/0   & 5    & 0.10    & 0.003 \\
    G3    & 100/0   & 8    & 0.15    & 0.004 \\
    G4    & 100/0   & 12   & 0.20    & 0.005 \\
    G5    & 100/0   & 15   & 0.30    & 0.006 \\
    \bottomrule
    \end{tabular}%
  \label{tab:G_parameters}
\end{table}

\subsection{Data of NYISO 1814-bus system}
\label{subsec:NYISO_system}

The NYISO system, consisting of 1814 buses (black dots), 2207 lines, 362 generators, and 33 wind farms (blue dots), is shown in Fig.~\ref{fig:NYISO_system}.
Colors of the transmission lines reflect power flows, with red denoting heavy flow and green indicating light flow. Specific parameters of this system can be found in \cite{NYISO_data}. Fig.~\ref{fig:NYISO_LMPs} shows the heat-map of LMPs at 4 p.m. on Aug. 28, 2018, which is a moment with particularly heavy load and high LMPs.

\begin{figure}[!t]
    \centering
    \includegraphics[width=0.9\linewidth]{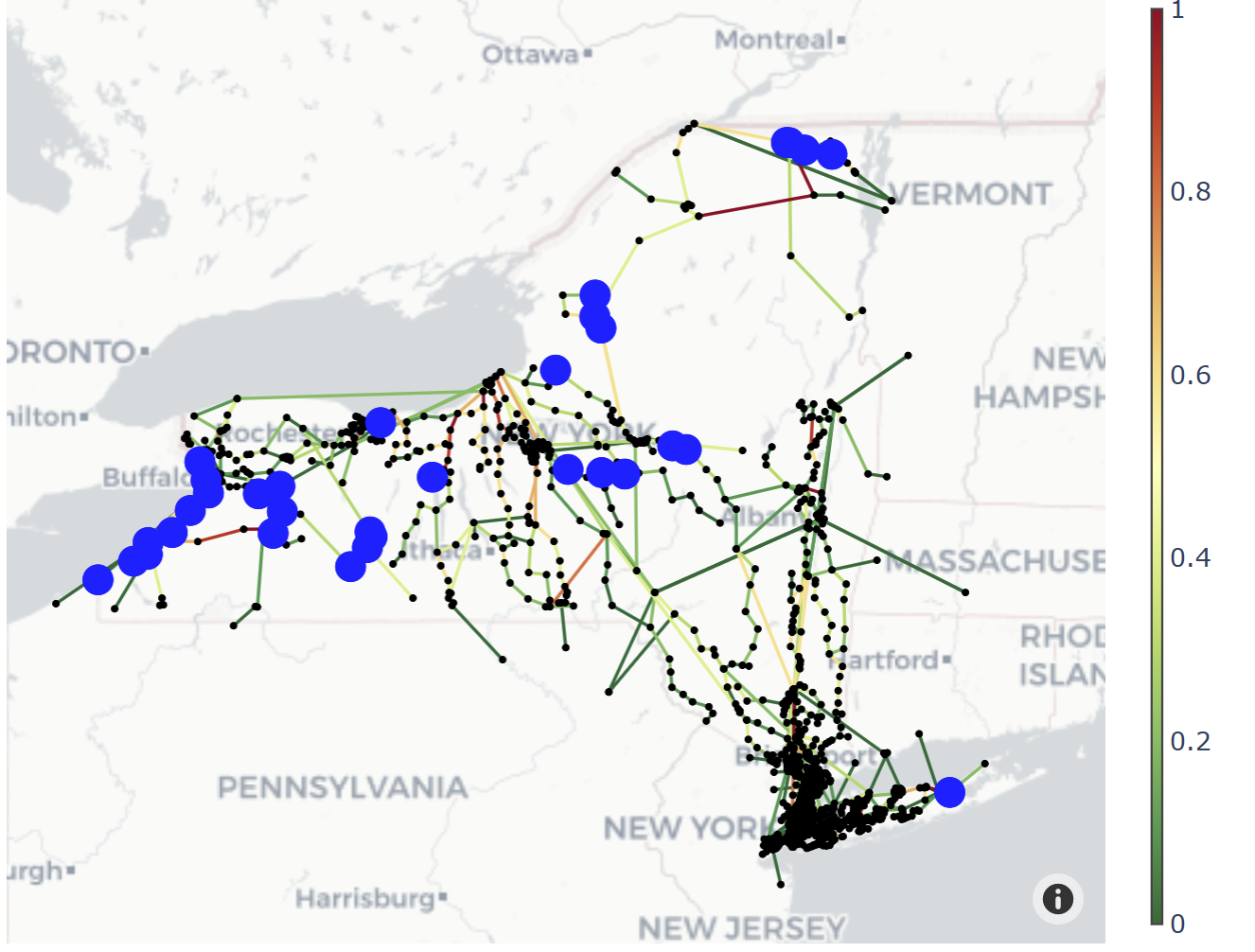}
    \caption{NYISO 1814-bus system.}
    \label{fig:NYISO_system}
\end{figure}

\begin{figure}[!t]
    \centering
    \includegraphics[width=0.9\linewidth]{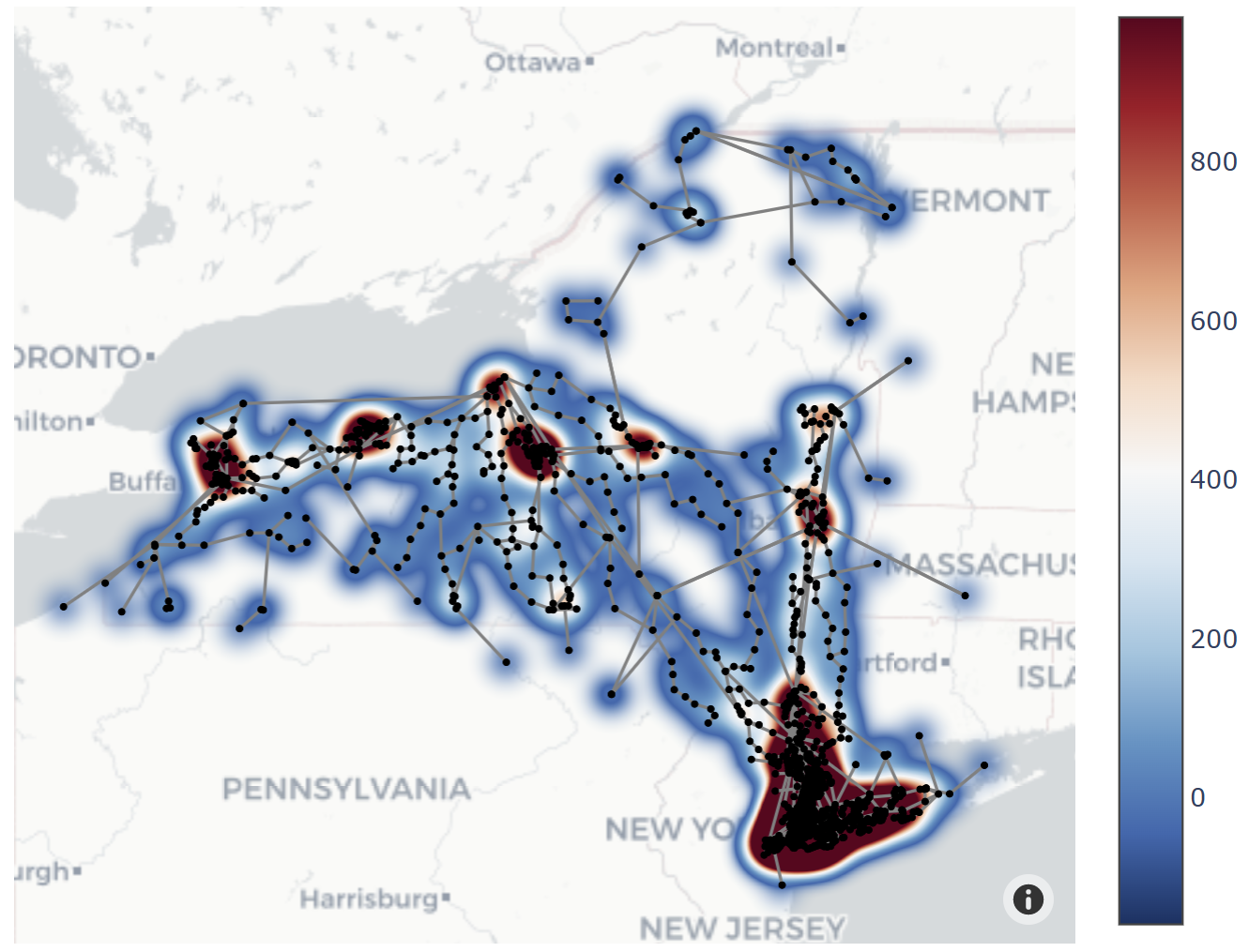}
    \caption{LMP heat-map of NYISO system at 4 p.m. on Aug. 28, 2018 (Unit of LMPs: \$/MWh).}
    \label{fig:NYISO_LMPs}
\end{figure}

\end{document}